\documentclass[11pt,reqno]{amsart}

\usepackage[
  letterpaper,
  margin=0.95in
]{geometry}
\usepackage[T1]{fontenc}
\usepackage{lmodern}
\usepackage{microtype}
\usepackage{mathtools}
\usepackage{amssymb}
\usepackage{xcolor}
\usepackage{graphicx}
\usepackage{placeins}
\usepackage{aliascnt}
\usepackage[hidelinks]{hyperref}
\hypersetup{
  pdfauthor={Simon Becker and Izak Oltman},
  pdftitle={Spectra of Non-Self-Adjoint Almost Mathieu Matrices and the Scottish Flag Operator},
  pdfsubject={Phase dependence and limiting eigenvalue distributions for unit-modulus almost Mathieu matrices},
  pdfkeywords={almost Mathieu matrix, unit-modulus coupling, tridiagonal matrix, nonnormal spectrum, limiting eigenvalue distribution}
}
\graphicspath{{figures/}{./}}
\numberwithin{equation}{section}
\allowdisplaybreaks
\newtheorem{theorem}{Theorem}[section]
\newaliascnt{proposition}{theorem}
\newtheorem{proposition}[proposition]{Proposition}
\aliascntresetthe{proposition}
\newaliascnt{lemma}{theorem}
\newtheorem{lemma}[lemma]{Lemma}
\aliascntresetthe{lemma}
\newaliascnt{corollary}{theorem}
\newtheorem{corollary}[corollary]{Corollary}
\aliascntresetthe{corollary}
\theoremstyle{remark}
\newaliascnt{remark}{theorem}

\aliascntresetthe{remark}

\usepackage[nameinlink,capitalize,noabbrev]{cleveref}
\crefname{theorem}{theorem}{theorems}
\Crefname{theorem}{Theorem}{Theorems}
\crefname{proposition}{Proposition}{Propositions}
\Crefname{proposition}{Proposition}{Propositions}
\crefname{lemma}{lemma}{lemmas}
\Crefname{lemma}{Lemma}{Lemmas}
\crefname{corollary}{corollary}{corollaries}
\Crefname{corollary}{Corollary}{Corollaries}
\crefname{remark}{remark}{remarks}
\Crefname{remark}{Remark}{Remarks}

\newcommand{\C}{\mathbb C}
\newcommand{\R}{\mathbb R}
\newcommand{\Z}{\mathbb Z}
\newcommand{\e}{\mathrm e}
\newcommand{\spec}{\operatorname{spec}}
\newcommand{\diag}{\operatorname{diag}}
\newcommand{\tr}{\operatorname{tr}}
\newcommand{\rank}{\operatorname{rank}}
\newcommand{\cX}{\mathcal X}
\newcommand{\cG}{\mathcal G}
\newcommand{\push}{_{\#}}
\newcommand{\EllK}{\mathsf K}

\title[Non-self-adjoint almost Mathieu matrices]{Spectra of Non-Self-Adjoint Almost Mathieu Matrices and the Scottish Flag Operator}

\author{Simon Becker}
\address{Department of Decision Sciences, Bocconi University, Via Roentgen 1, 20136 Milan, Italy}
\email{simon.becker@unibocconi.it}
\author{Izak Oltman}
\address{Department of Mathematics, Northwestern University, Evanston, Illinois 60208, USA}
\email{ioltman@northwestern.edu}

\date{}

\begin{document}

\begin{abstract}
For \(N\ge3\) and a potential phase
\(\vartheta\in\R\), we study the non-self-adjoint almost Mathieu
matrix that we obtain by multiplying the discrete Laplacian by a complex phase with angle \(\varphi\in\R\),
\[
 A_N(\varphi,\vartheta)
 = \e^{i\varphi}\frac{S+S^{-1}}2+\diag\!\left(\cos\!\left(\frac{2\pi j}{N}+\vartheta\right)\right)_{j\in\Z/N\Z},
\]
where \(S e_j=e_{j+1}\) is the periodic shift on \(\C^N\).  We derive a Chambers' formula and isolate the part \(Q_{N,\varphi}\) of the characteristic polynomial that only depends on $N$ and $\varphi$ but not $\vartheta$ or a change of boundary conditions of the shift operator.  We then show, for every \(N\), that the zeros of
\(Q_{N,\varphi}\) lie on the two perpendicular lines
\[
 \e^{i\varphi/2}\R\cup
 \e^{i(\varphi/2+\pi/2)}\R.
\]
For even \(N\), the same property holds for the
matrices $A_N(\varphi,\vartheta)$ with \(\vartheta\in2\pi\Z/N\), and we compute their limiting
eigenvalue measure explicitly.  For
\(\varphi\in[-\pi,\pi]\), the eigenvalue distribution approximates elliptic-integral densities with masses
\(1-|\varphi|/\pi\) and \(|\varphi|/\pi\), and maximal radii
\(2|\cos(\varphi/2)|\) and
\(2|\sin(\varphi/2)|\), respectively.
At \(\varphi=\pi/2\), the central polynomial \(Q_{N,\varphi}\) factors into positive quartic
factors.  This proves that the Scottish flag matrix, after Trefethen and Chapman, has its spectrum on the two
diagonal lines of the saltire, see Figure \ref{fig:saltire}.
\end{abstract}

\maketitle

\section{Introduction and main results}\label{sec:introduction}
\begin{figure}
    \includegraphics[width=5.5cm]{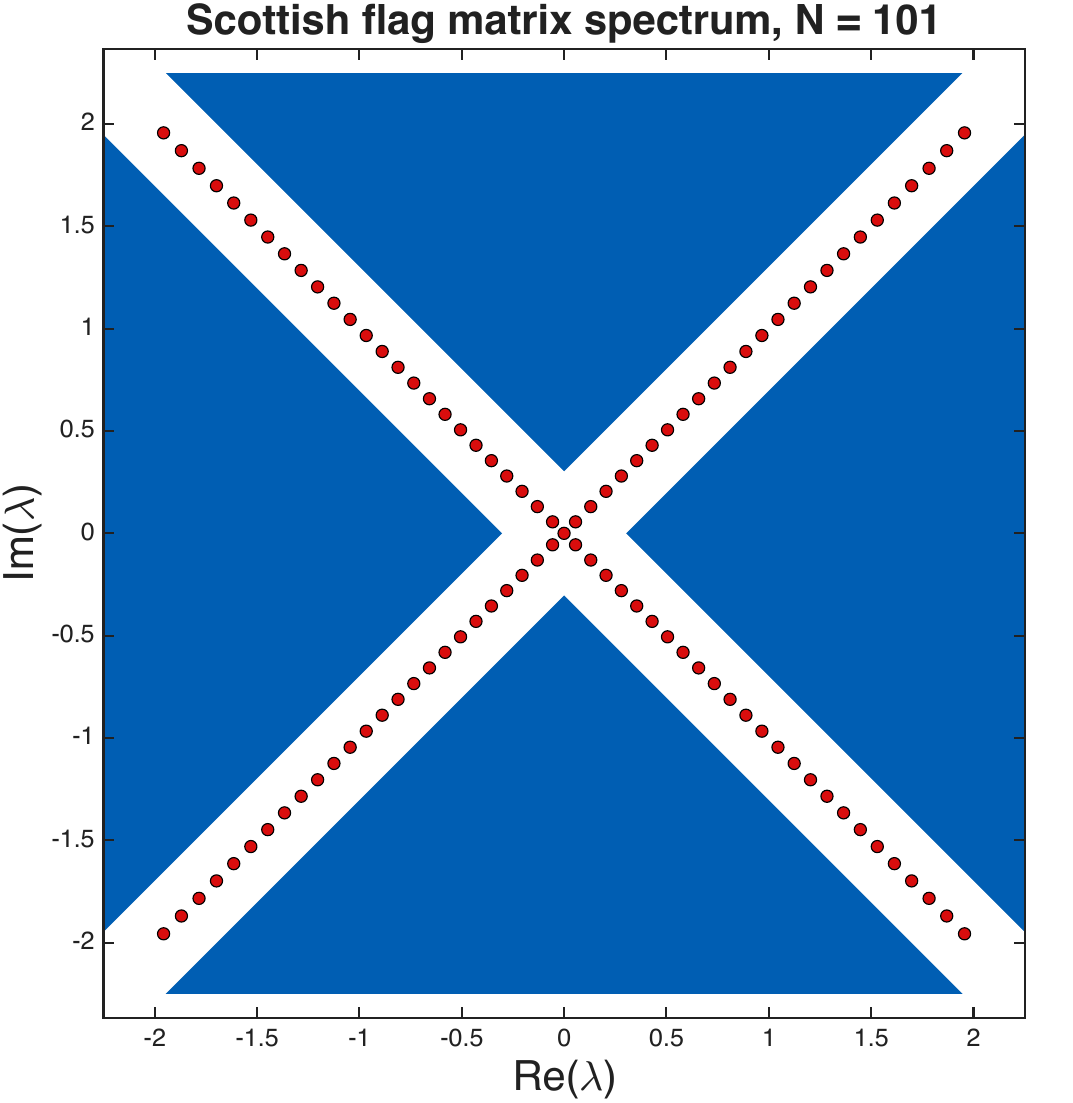}
    \caption{The eigenvalues (red dots) of \eqref{eq:BN} for $N=101$, shown against the flag of Scotland.}
    \label{fig:saltire}
\end{figure}
Let \(S e_j=e_{j+1}\) on \(\C^N\), with indices in \(\Z/N\Z\), i.e., periodic boundary conditions, and set
\[
 C_N:=\frac{S+S^{-1}}2,
 \text{ and }
 D_N(\vartheta):=
 \diag\!\left(\cos\!\left(\frac{2\pi j}{N}+\vartheta\right)\right)_{j\in\Z/N\Z}.
\]
For a coupling angle \(\varphi\in\R\), we study
\[
 A_N(\varphi,\vartheta):=\e^{i\varphi}C_N+D_N(\vartheta).
\]
Our motivation for this work is due to an observation by Trefethen and Chapman who
introduced the nonnormal matrix
\begin{equation}
\label{eq:BN}
 B_N=S-S^{-1}+2\diag(\sin(2\pi j/N))_{j\in\Z/N\Z}
\end{equation}
and used it to illustrate large two-dimensional pseudospectra
\cite{TrefethenChapman2004}, i.e., small perturbations leading to significant changes in the spectral structure.  
They numerically compute in \cite[Figure~1.1]{TrefethenChapman2004} the \(N=101\)
spectrum on the union of two orthogonal lines through the origin, inclined
at \(45^\circ\) to the coordinate axes, see Figure \ref{fig:saltire}. This spectral picture led to the name \emph{Scottish flag matrix}.  It had, however, only been observed numerically, and not been rigorously established.

The related periodic matrix \(A_N(\pi/2,0)\) also appeared in \cite{Oltman2023}; see also
\cite{BorthwickUribe,Vogel}.

To cover more general boundary conditions than just periodic ones, consider \(|\tau|=1\), write \(\tau=\e^{i\kappa}\) with \(\kappa\in\R\), and
let \(S_\tau\) denote the shift with boundary phase \(\tau\), the twist, i.e.,
\[
 S_\tau e_j=e_{j+1}\text{ for }0\le j<N-1,
 \qquad S_\tau e_{N-1}=\tau e_0.
\]
We can then define the slightly more general family
\[
 A_N^\tau(\varphi,\vartheta)
 :=\e^{i\varphi}\frac{S_\tau+S_\tau^{-1}}2+D_N(\vartheta).
\]
Thus \(A_N(\varphi,\vartheta)=A_N^1(\varphi,\vartheta)\).

Conjugating by \(G_N=\diag(1,i,i^2,\ldots,i^{N-1})\) then identifies the Scottish flag operator of Trefethen and Chapman with our family in this article. Indeed,
\begin{equation}
\label{eq:SF}
 -G_N^{-1}B_NG_N
 =2A_N^{i^N}(\pi/2,\pi/2),
\end{equation}
i.e., value \(\varphi=\pi/2\) is the one
that recovers the Scottish flag matrix \(B_N\).  A few cases are particularly simple:

If \(4\mid N\), then
\(\tau=1\), i.e., the shift is equipped with periodic boundary conditions, and translation gives a matrix unitarily equivalent to
\(2A_N(\pi/2,0)\).  If \(N\equiv2\pmod4\), then \(\tau=-1\), i.e., the shift is equipped with anti-periodic boundary conditions.
The determinant identity in \Cref{thm:intro-arbitrary-phase} shows that, for
every even \(N\), the matrix arising from \(B_N\) has the same
characteristic polynomial as \(A_N(\pi/2,0)\), after the scaling above. Thus, for even $N$, the spectral problem reduces to the study of periodic boundary conditions.
When \(N\equiv2\pmod4\), their Jordan blocks at zero are nevertheless different, even though their spectra coincide, i.e. it is not a unitary equivalence.

Our matrices $A_N^{\tau}$ are (in general) a non-normal analogue of rational almost Mathieu matrices; for background on the rational-frequency problem, see
\cite{Harper1955,Hofstadter1976,AubryAndre1980,AvronSimon1983,
Chambers1965,LamoureuxMingo2007,JitomirskayaKonstantinovKrasovsky2022}.  The corresponding irrational-frequency problem is substantially harder in the non-self-adjoint setting, where the simple quantitative perturbation bounds available in the self-adjoint case are absent.  We do not pursue it here.  
A separate work in preparation by the second author and Klopp studies a
much more general class of discrete Mathieu-type operators
\cite{OltmanKloppInPrep}.  Here we focus on exact finite-dimensional identities and on the explicit limiting measure for even periodic matrices, mostly when
\(\vartheta\in2\pi\Z/N\).  The first result gives conditions under which the spectrum is confined to two perpendicular lines.

\subsection{Spectra confined to perpendicular lines}

The coupling phase $\varphi$ and the potential phase $\vartheta$ have very
different roles.  For even \(N\), $\tau=1$, i.e., periodic boundary conditions, and \(\vartheta\in2\pi\Z/N\), we show that the spectrum lies on two perpendicular lines, as illustrated in \Cref{fig:spectrum-vs-phi}.  To parametrize these lines, let \(\alpha\in\R\), and set
\[
 \cX_\alpha:=\e^{i\alpha}\R\cup\e^{i(\alpha+\pi/2)}\R.
\]
\begin{figure}[t]
    \centering
    \includegraphics[width=\textwidth]{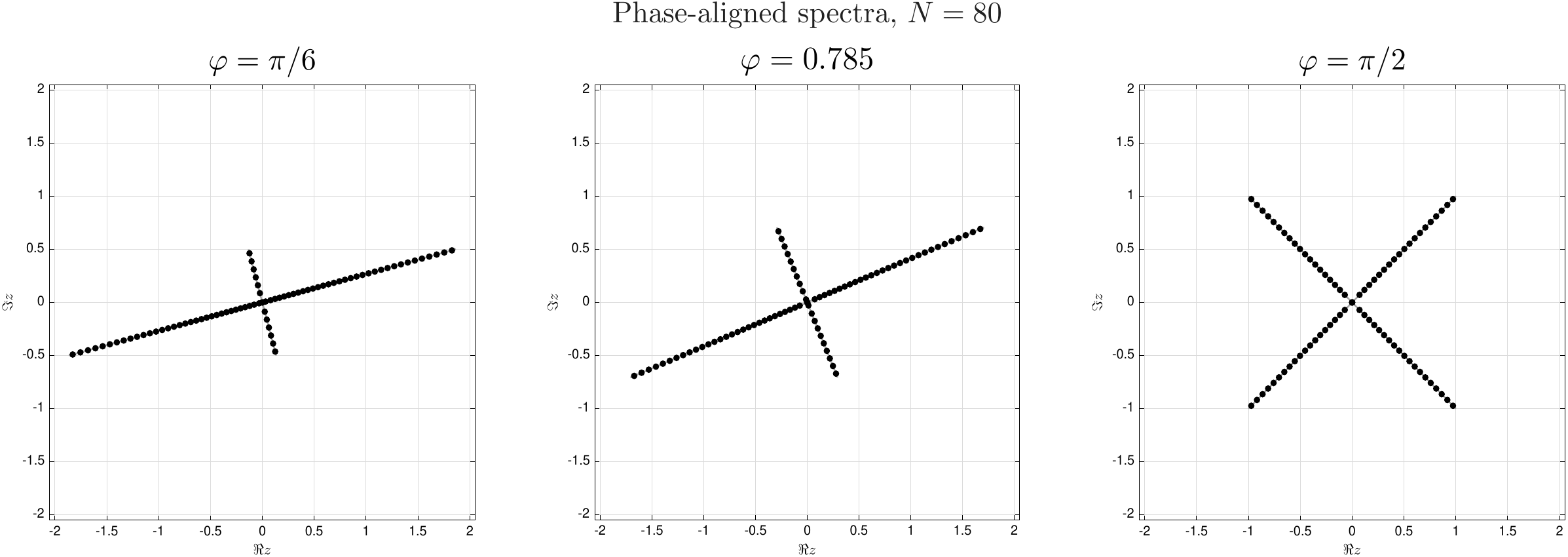}
    \caption{Spectra of \(A_{80}(\varphi,0)\) for
    \(\varphi=\pi/6,\pi/4,\pi/2\) (left to right), illustrating
    \Cref{thm:intro-unit-modulus}.  The solid and dashed guides are the two
    lines, truncated at the bounds
    \(2|\cos(\varphi/2)|\) and \(2|\sin(\varphi/2)|\), respectively.}
    \label{fig:spectrum-vs-phi}
\end{figure}
Our first main result is then:
\begin{theorem}
\label{thm:intro-unit-modulus}
Let \(N\ge4\) be even, \(\varphi\in\R\), and
\(\vartheta\in2\pi\Z/N\).  Then
\begin{equation}
 \label{eq:rotated-cross-intro}
 \spec A_N(\varphi,\vartheta)\subset \cX_{\varphi/2}.
\end{equation}
More precisely, if \(z\in\spec A_N(\varphi,\vartheta)\), then we have the upper bounds for the lines
\[
 z\in\e^{i\varphi/2}\R
 \quad\Longrightarrow\quad |z|\le2|\cos(\varphi/2)|\text{ and }
 z\in\e^{i(\varphi/2+\pi/2)}\R
 \quad\Longrightarrow\quad |z|\le2|\sin(\varphi/2)|.
\]
\end{theorem}

In the next subsection, we will see that the restriction on $\vartheta$ is necessary. Some restriction on $N$ is also necessary, as the following example shows:
When \(N=5\) and
\((\varphi,\vartheta)=(\pi/2,0)\),
\[
 \det\bigl(zI-A_5(\pi/2,0)\bigr)
 =z^5+\frac{5(1+\sqrt5)}{32}z-\frac{1+i}{16},
\]
but this polynomial has a root off the saltire $\cX_{\pi/4}.$

The first theorem concerns periodic boundary conditions, $\tau=1$. To treat the Scottish flag matrix, we must allow general boundary twists.

\subsection{Phase dependence, Chambers' formula, and the Scottish flag effect}

For arbitrary \(\vartheta\), the following determinant identity shows that the potential phase
only changes the constant term of the characteristic polynomial. Thus, the
eigenvalues may no longer be located on the distinguished lines, as shown
in \Cref{fig:theta-dependence}. We also isolate the role of the boundary condition $\tau$ of the shift to treat the Scottish flag matrix. 

For \(\tau=\e^{i\kappa}\), write $\chi_{N,\varphi,\vartheta}^{\tau}(z)
 :=\det\bigl(zI-A_N^\tau(\varphi,\vartheta)\bigr)$ for the characteristic polynomial. A decomposition separating the $(\tau,\vartheta)$-dependent constant term from a polynomial $Q_{N,\varphi}(z)$ is called a Chambers' formula in the theory of the AMO; see \eqref{eq:283}. The constant term also depends on $N$ and $\varphi$, while $Q_{N,\varphi}$ is independent of $\tau$ and $\vartheta$. We generalize this formula here to our setting in the following theorem:
\begin{theorem}
\label{thm:intro-arbitrary-phase}
Let \(N\ge3\) and \(\varphi,\vartheta\in\R\).  Set
\(\varepsilon_N:=N\bmod2\in\{0,1\},
 \delta_N:=\frac{1-\varepsilon_N}{2},
 \omega:=\e^{2\pi i/N},\)
and define $\psi_k(j):=\frac1{\sqrt N}
 \exp\!\left(\frac{\pi i j(j-\varepsilon_N)}N\right)\omega^{jk},$ for $ j,k\in\Z/N\Z.$
Let \(\widehat A_N(\varphi,\vartheta)\) be the matrix of
\(A_N(\varphi,\vartheta)\) in the orthonormal basis
\((\psi_k)_{k\in\Z/N\Z}\), and set $\rho_{k,N}(\varphi,\vartheta)
 :=\widehat A_N(\varphi,\vartheta)_{k+1,k}
   \widehat A_N(\varphi,\vartheta)_{k,k+1}.$
Then
\begin{equation}
 \label{eq:general-phase-product-intro}
 \rho_{k,N}(\varphi,\vartheta)
 =\frac{\e^{i\varphi}}2
 \left(\cos\varphi+
 \cos\!\left(\frac{2\pi(k+\delta_N)}N-\vartheta\right)\right).
\end{equation}
Moreover, there is a monic polynomial \(Q_{N,\varphi}(z)\), independent of
\(\vartheta\), such that for every
\(\tau=\e^{i\kappa}\)
\begin{equation}\label{eq:283}
 \chi_{N,\varphi,\vartheta}^{\tau}(z)
 =Q_{N,\varphi}(z)
 -2^{1-N}\cos(N\vartheta)
 -2^{1-N}\e^{iN\varphi}\cos\kappa.
\end{equation}
This allows us to compare different boundary twists $\tau$ and potential phases $\vartheta$:
\begin{equation}
 \label{eq:twisted-Chambers-intro}
 \chi_{N,\varphi,\vartheta}^{\tau}(z)
 =\chi_{N,\varphi,0}^{1}(z)
 +2^{1-N}\bigl(1-\cos(N\vartheta)\bigr)
 +2^{1-N}\e^{i N\varphi}\bigl(1-\cos\kappa\bigr).
\end{equation}
Thus all periodic matrices with \(\vartheta\in2\pi\Z/N\) are isospectral for every \(N\).
When \(\varphi=\pi/2\) and $N$ is even
\begin{equation}
 \label{eq:twist-cancellation-intro}
 \chi_{N,\pi/2,\pi/2}^{i^N}(z)
 =\chi_{N,\pi/2,0}^{1}(z),
\end{equation}
whereas for $N$ odd
\begin{equation}
 \label{eq:odd-central-Chambers-intro}
 \chi_{N,\pi/2,\pi/2}^{i^N}(z)=Q_{N,\pi/2}(z).
\end{equation}
\end{theorem}
The identities \eqref{eq:twist-cancellation-intro} and  \eqref{eq:odd-central-Chambers-intro} are the important identities for the Scottish flag matrix. The first one shows that the even $N$ case is fully resolved by Theorem \ref{thm:intro-unit-modulus}. The second one shows that to understand $N$ odd, we need to understand the roots of $Q_{N,\pi/2}$. 

This is the content of the next theorem:

\begin{theorem}
\label{thm:intro-central-cross}
Let \(N\ge3\), \(\varphi\in\R\), and put $\ell_N:=\left\lfloor\frac N2\right\rfloor,$ and $\varepsilon_N:=N \mod 2 \in\{0,1\}.$
For the polynomial \(Q_{N,\varphi}\) in
\eqref{eq:283}, there is a monic real-rooted polynomial
\(R_{N,\varphi}\in\R[t]\) of degree \(\ell_N\) such that
\begin{equation}
\label{eq:central-real-rooted-form}
 Q_{N,\varphi}(z)
 =\e^{i\ell_N\varphi}z^{\varepsilon_N}
 R_{N,\varphi}\!\left(\e^{-i\varphi}z^2\right).
\end{equation}
Consequently,
\begin{equation}
\label{eq:central-cross-intro}
 Q_{N,\varphi}(z)=0
 \quad\Longrightarrow\quad
 z\in\cX_{\varphi/2}.
\end{equation}
More generally, if \(\tau=\e^{i\kappa}\) and
\begin{equation}
\label{eq:central-fibre-condition}
 \cos(N\vartheta)+\e^{iN\varphi}\cos\kappa=0,
\end{equation}
then
\[
 \chi_{N,\varphi,\vartheta}^{\tau}=Q_{N,\varphi},
 \qquad
 \spec A_N^\tau(\varphi,\vartheta)\subset\cX_{\varphi/2}.
\]
Such a fibre exists for every \(N\) and \(\varphi\); one may always take $\vartheta=\frac{\pi}{2N},
 \tau=i.$
Every zero \(z\) of \(Q_{N,\varphi}\) satisfies
\[
 \begin{aligned}
 z\in\e^{i\varphi/2}\R
 \quad\Longrightarrow\quad
 |z|\le2\left|\cos\frac{\varphi}{2}\right| \text{ and }
 z\in\e^{i(\varphi/2+\pi/2)}\R
 \quad\Longrightarrow\quad
 |z|\le2\left|\sin\frac{\varphi}{2}\right|.
 \end{aligned}
\]
\end{theorem}

We thus conclude for the Scottish flag matrix: 

\begin{corollary}
    For $N \ge 3$, we find $\spec(B_N) \subset \cX_{\pi/4}.$
\end{corollary}

For \(\varphi=\pi/2\) and \(N\equiv2\pmod4\), we have from \eqref{eq:central-fibre-condition} 
\[
 \vartheta\in2\pi\Z/N
 \quad\Longrightarrow\quad
 \spec A_N(\pi/2,\vartheta)\subset\cX_{\pi/4},
\]
whereas from the factorization of the characteristic polynomial at \(\vartheta=0\) in 
\Cref{thm:intro-scottish} and \eqref{eq:twisted-Chambers-intro}
\[
 \vartheta\notin2\pi\Z/N
 \quad\Longrightarrow\quad
 \spec A_N(\pi/2,\vartheta)\cap\cX_{\pi/4}=\varnothing
 \text{ for }N\equiv2\!\!\pmod4.
\]
The exclusion result has a quantitative companion.  Put
\[
 d_N(\vartheta):=\operatorname{dist}\!\left(
 \vartheta,\frac{2\pi}{N}\mathbb Z\right)\in[0,\pi/N]
\text{ and }
 \gamma_N(\vartheta):=
 \frac12\sqrt{\sin d_N(\vartheta)\,
 \sin\!\left(\frac{2\pi}{N}-d_N(\vartheta)\right)}.
\]
Then every \(z\in\spec A_N(\pi/2,\vartheta)\) satisfies
\[
 |\Re z^2|\le\gamma_N(\vartheta)
 \le\frac12\sin\frac{\pi}{N} \text{ and therefore }
 \operatorname{dist}(z,\cX_{\pi/4})
 \le \sqrt{\frac{\gamma_N(\vartheta)}2}
 \le\frac12\sqrt{\sin\frac{\pi}{N}}.
\]
Thus the whole periodic spectrum is \(\mathcal  O(N^{-1/2})\)-close to the two diagonal lines,
uniformly in the phase. 
See \Cref{cor:periodic-phase-dichotomy}.

\begin{figure}[t]
    \centering
    \includegraphics[width=0.98\textwidth]{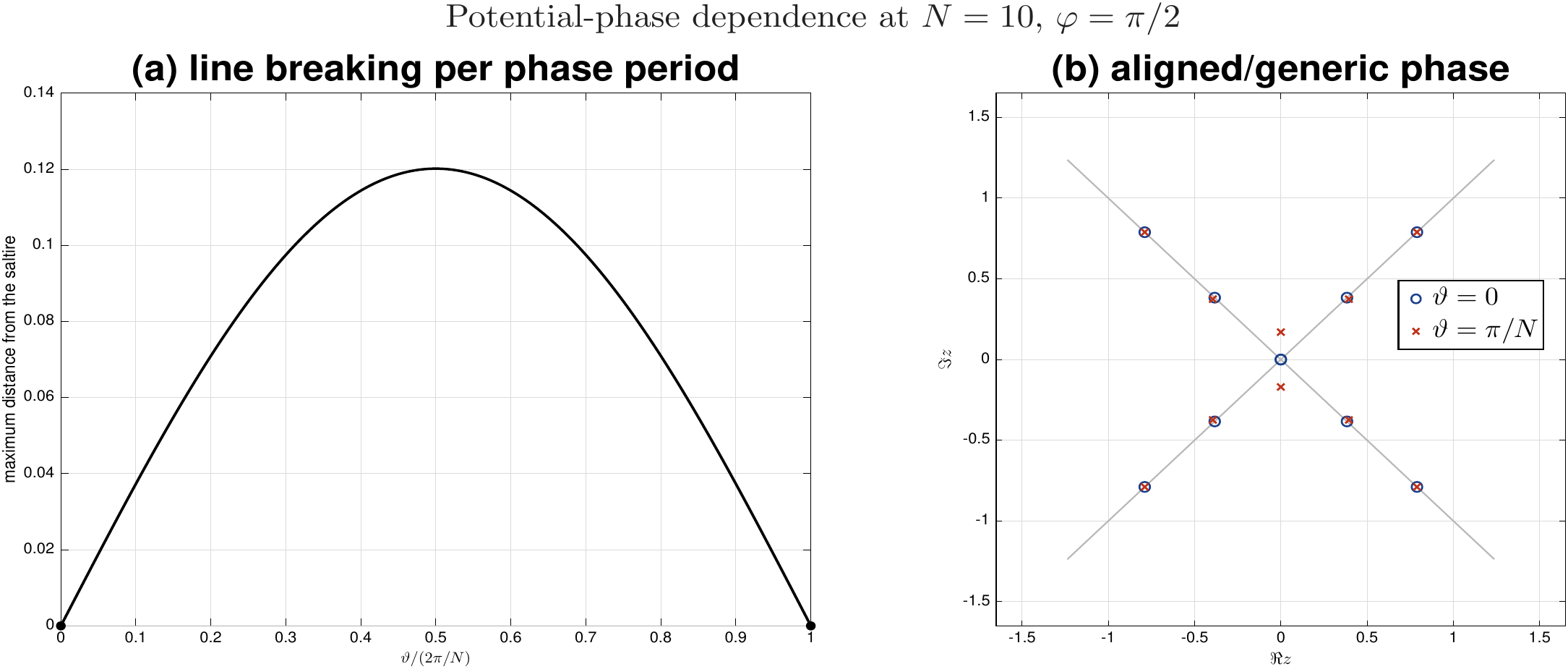}
    \caption{Potential-phase dependence at \(N=10\) and
    \(\varphi=\pi/2\), illustrating \Cref{thm:intro-arbitrary-phase}.
    Let \(\mathcal S=\cX_{\pi/4}\) and define
    \(d_{\mathcal S}(z):=\min\{|\Im(\e^{-\pi i/4}z)|,
    |\Im(\e^{-3\pi i/4}z)|\}\).  The left figure (a) plots
    \(\max_{z\in\spec A_N(\pi/2,\vartheta)}d_{\mathcal S}(z)\)
    at 101 equally spaced values of \(\vartheta\) over one phase period
    \([0,2\pi/N]\).  The right figure (b) compares \(\vartheta=0\) with \(\vartheta=\pi/N\).}
    \label{fig:theta-dependence}
\end{figure}

\subsection{The limiting distribution}
Our last main result gives the distribution of eigenvalues on the two lines of $\cX_{\alpha}.$ 

For this, we use the complete elliptic integral of the first kind in the modulus
convention
\begin{equation*}
 \EllK(k):=\int_0^{\pi/2}\frac{d\alpha}
 {\sqrt{1-k^2\sin^2\alpha}},
 \qquad 0\le k<1 \text{ with }\EllK(1)=+\infty.
\end{equation*}
 For \(\alpha\in\R\),
let \(R_\alpha\colon\R\to\C\) be the rotation $R_\alpha(t):=\e^{i\alpha}t.$

For a measure \(\nu\) and a map \(T\), write \(T\push\nu\) for the measure defined by
\[
 (T\push\nu)(E)=\nu(T^{-1}(E)).
\]
We also write \(\delta_z\) for the measure concentrated at the single point \(z\).  For
\(\varphi\in[-\pi,\pi]\), define
\[
 a_{\varphi,+}:=\cos\frac{|\varphi|}{2} \text{ and }
 a_{\varphi,-}:=\sin\frac{|\varphi|}{2}.
\]
\begin{theorem}
\label{thm:intro-limit}
Fix \(\varphi\in[-\pi,\pi]\).  For even \(N\), define the probability
measure
\[
 \mu_{N,\varphi}:=\frac1N
 \sum_{z\in\spec A_N(\varphi,0)}\delta_z,
\]
where eigenvalues are counted with algebraic multiplicity.  Define
\begin{equation}
\label{eq:general-line-densities}
 g_{\varphi,\pm}(t):=
 \begin{cases}
 \displaystyle
 \frac{1}{\pi^2}
 \EllK\!\left(
 \sqrt{a_{\varphi,\pm}^2-\frac{t^2}{4}}
 \right),
 & |t|<2a_{\varphi,\pm},\\[2mm]
 0,
 & |t|\ge 2a_{\varphi,\pm}.
 \end{cases}
\end{equation}
Then, as \(N\to\infty\) through even integers, weakly as probability measures on \(\C\),
\begin{equation}
\label{eq:general-explicit-limit}
 \mu_{N,\varphi}\Longrightarrow
 (R_{\varphi/2})\push\bigl(g_{\varphi,+}(t)\,dt\bigr)
 +(R_{\varphi/2+\pi/2})\push\bigl(g_{\varphi,-}(t)\,dt\bigr).
\end{equation}
The two measures on the two lines have masses
\begin{equation}
\label{eq:line-masses-intro}
 \int_\R g_{\varphi,+}(t)\,dt=1-\frac{|\varphi|}{\pi},
 \qquad
 \int_\R g_{\varphi,-}(t)\,dt=\frac{|\varphi|}{\pi}.
\end{equation}
If \(a_{\varphi,\pm}>0\), the support of the corresponding measure on that line is
\([-2a_{\varphi,\pm},2a_{\varphi,\pm}]\).  When
\(a_{\varphi,\pm}=0\), that measure on that line vanishes and its support is empty.
Let \(\mu_\varphi\) denote the measure on the right-hand side of
\eqref{eq:general-explicit-limit}.  For every polynomial \(P(z,\bar z)\),
\begin{equation}
\label{eq:general-polynomial-rate}
 \frac1N\sum_{z\in\spec A_N(\varphi,0)}P(z,\bar z)
 =\int_\C P(z,\bar z)\,d\mu_\varphi(z)+\mathcal O_{P,\varphi}(N^{-1}).
\end{equation}
In particular, for the periodic matrix at \(\varphi=\pi/2\), one has
\[
 a_{\pi/2,+}=a_{\pi/2,-}=\frac1{\sqrt2},
 \qquad
 g_{\pi/2,+}=g_{\pi/2,-}.
\]
\end{theorem}
The extension to all dimensions is given in
\Cref{cor:full-scottish-limit}.
\begin{figure}[t]
    \centering
    \includegraphics[width=\textwidth]{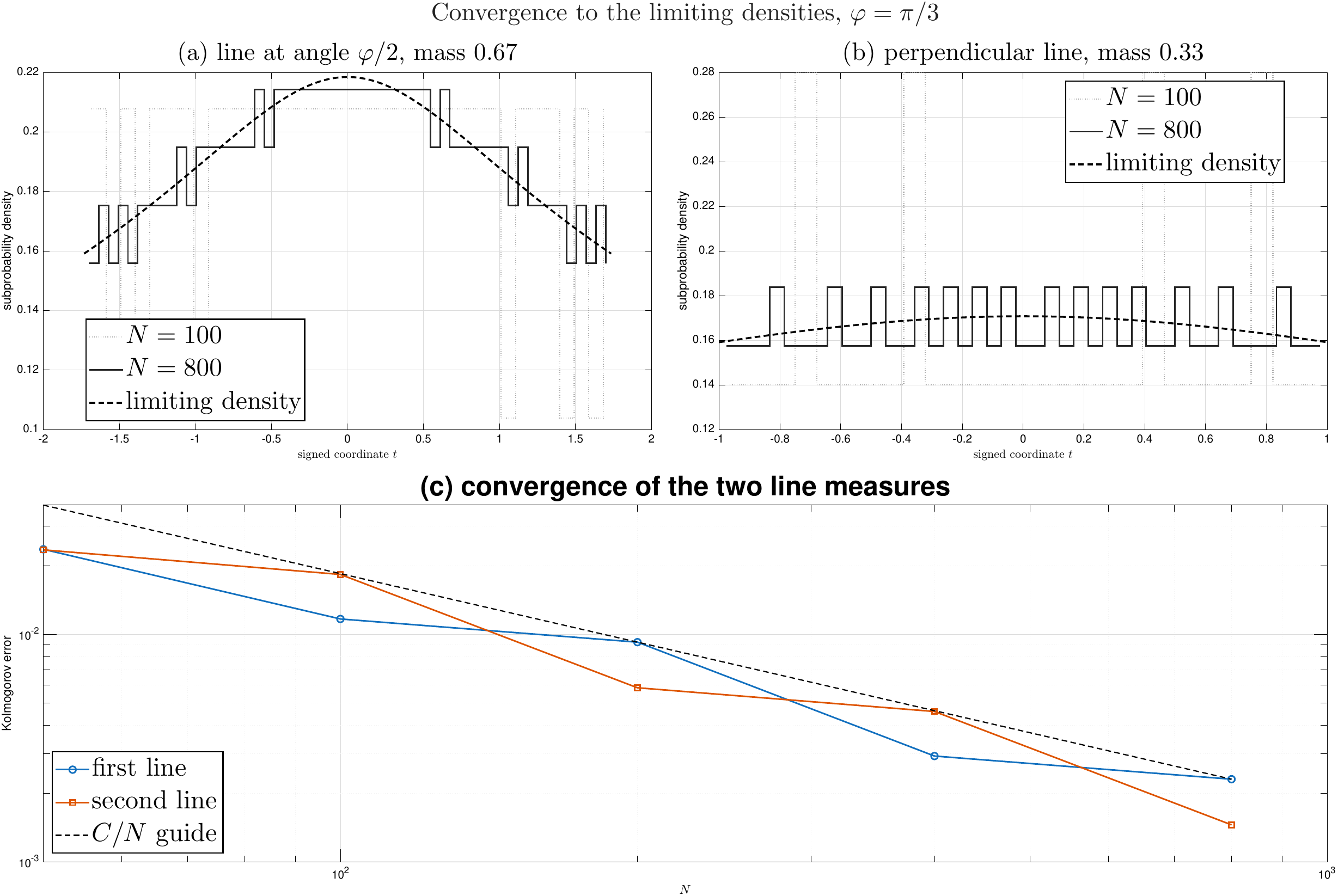}
    \caption{Eigenvalue distribution of $A_N(\varphi,0)$ for \(\varphi=\pi/3\).  Subfigure (a) shows a 48-bin
    histogram of the coordinates on the line at complex angle $\pi/6$,
    together with the limiting subprobability density \(g_{\pi/3,+}\).
    Subfigure (b) shows the same for the perpendicular line
    \(g_{\pi/3,-}\).
    Subfigure (c) compares
    the radial distribution function \(F_N\) with the limiting radial distribution function \(H\), for both perpendicular lines individually.  The subfigure shows the numerical
    (Kolmogorov) error \(\|F_N-H\|_\infty\) for both even
    congruence classes.  The \(N^{-1}\) line is a numerical reference
    slope; \Cref{thm:intro-limit} proves an \(N^{-1}\) polynomial-moment
    error, not a Kolmogorov bound.}
    \label{fig:density-convergence}
\end{figure}

\subsection*{Organization}

The paper is organized as follows:
\begin{itemize}
\item \Cref{sec:unit-modulus} develops the parity-adapted basis, the
zero-diagonal path criteria, and the even aligned-periodic theorem.
\Cref{sub:phase-opening} proves \Cref{thm:intro-central-cross}.

\item \Cref{sec:scottish} treats \(\varphi=\pi/2\).
\Cref{prop:central-scottish-factorization} factors the central polynomial;
\Cref{thm:intro-scottish} handles the even periodic matrix, and
\Cref{thm:odd-scottish} handles the odd Scottish flag matrix.  They are
combined in the all-dimensional \Cref{thm:original-scottish}.

\item \Cref{sec:distribution} proves \Cref{thm:intro-limit} and the
full-sequence Scottish corollary \Cref{cor:full-scottish-limit}.
\item Appendices \ref{app:mirrored}, \ref{app:scottish-zero}, and
\ref{app:antiperiodic-zero} contain the determinant and zero-eigenvalue
computations.
\end{itemize}
\textbf{Connection to torus quantization}
Write $\tau=e^{i\kappa}$ and introduce the diagonal gauge
\[
U_\kappa=\operatorname{diag}\bigl(e^{-ij\kappa/N}\bigr)_{j\in\mathbb Z/N\mathbb Z}.
\]
Then
\[
U_\kappa^{-1}S_\tau U_\kappa=e^{i\kappa/N}S,
\]
so that
\[
U_\kappa^{-1}A_N^\tau(\varphi,\vartheta)U_\kappa
=
D_N(\vartheta)
+\frac{e^{i\varphi}}2
\left(
e^{i\kappa/N}S+e^{-i\kappa/N}S^{-1}
\right).
\]
Since $D_N(\vartheta)$ is the torus quantization of
$\cos(x+\vartheta)$ and
\[
\frac12\left(
e^{i\kappa/N}S+e^{-i\kappa/N}S^{-1}
\right)
\]
is the quantization of $\cos(\xi+\kappa/N)$, we obtain
\[
A_N^\tau(\varphi,\vartheta)
\simeq
\operatorname{Op}\!\left[
\cos(x+\vartheta)
+e^{i\varphi}\cos\!\left(\xi+\frac{\kappa}{N}\right)
\right],
\]
where $\simeq$ denotes unitary equivalence. Thus the boundary twist
$\tau$ may equivalently be viewed as an $N^{-1}\kappa$ shift of the
momentum variable in the periodic quantization.

We emphasize that in the microlocal literature, see for instance \cite{CZ10}, it is usually the symbol $\cos(x)+i\cos(\xi)$ that is associated with the Scottish flag operator. This operator corresponds to our matrix $A_N(\pi/2,0)$ and does not have its spectrum on the cross for general odd $N$, unlike the original operator considered by Chapman-Trefethen \cite{TrefethenChapman2004}.

\subsection*{Future Directions:} 
There are at least three natural generalizations of this theory that seem worth pursuing and are of different level of difficulty.
\begin{itemize}
\item An immediate generalization is to consider general rational frequencies $p/N$ rather than only $1/N.$ Although our family of operators $A_N$ has frequency \(1/N\) in the cosine potential, nonprimitive rational
frequencies reduce exactly to primitive ones.  For \(1\le p<N\), set
\[
 D_{N,p}(\vartheta):=
 \diag\!\left(\cos\!\left(\frac{2\pi p j}{N}+\vartheta\right)\right)_{j\in\Z/N\Z} \text{ and }
 A_{N,p}^{\tau}(\varphi,\vartheta)
 :=
 \e^{i\varphi}\frac{S_\tau+S_\tau^{-1}}2+D_{N,p}(\vartheta).
\]
Write
\[
 d:=\gcd(p,N),\qquad p=da,\qquad N=dq,\qquad \gcd(a,q)=1.
\]
Then the primitive denominator \(q\), rather than the ambient dimension
\(N\), controls the finite spectral geometry. As a first result one should be able to then show:

\textbf{Theorem}
For every \(|\tau|=1\),
\[
 A_{N,p}^{\tau}(\varphi,\vartheta)
 \simeq
 \bigoplus_{\zeta^d=\tau}A_{q,a}^{\zeta}(\varphi,\vartheta),
\]
where \(A_{q,a}^{\zeta}\) is the \(q\)-dimensional matrix with frequency
\(a/q\) and boundary twist \(\zeta\).  In particular,
\[
 \spec A_{N,p}^{\tau}(\varphi,\vartheta)
 =
 \bigcup_{\zeta^d=\tau}
 \spec A_{q,a}^{\zeta}(\varphi,\vartheta),
\]
with algebraic multiplicity and analyze this more complicated family, where it seems that many more spectra are possible.

\item In addition, one might try to understand irrational frequencies and study the corresponding discrete operator on $\ell^2(\mathbb Z)$. There seem to be many interesting questions about the nature of the spectrum, whether the spectrum lies on a finite union of lines, what is the spectral type etc.. 

\item The mechanism behind the Scottish flag phenomenon suggests a broader class of finite-dimensional non-self-adjoint models with higher rotational symmetries.  In the present setting, the quadratic-phase transform produces a $\mathbb Z_2$-graded, zero-diagonal matrix, so that squaring reduces the spectral problem to a real symmetric Jacobi matrix; the resulting relation $z^2\in e^{i\phi}\mathbb R$ is responsible for the two-line structure, while an additional reversal symmetry yields the fourfold symmetry at $\phi=\pi/2$.  A natural extension is to replace this $\mathbb Z_2$ grading by a cyclic $\mathbb Z_m$ grading, so that $A^m$ reduces to a phase multiple of a self-adjoint operator.  One would then expect
\[
    \operatorname{spec}(A)
    \subset
    \bigcup_{j=0}^{m-1}
    e^{i(\alpha+j\pi/m)}\mathbb R,
\]
with an additional sign-reversal symmetry potentially producing a $2m$-fold rotationally invariant spectrum and characteristic polynomials of the form
\[
    \det(z-A)=z^r\prod_j\bigl(z^{2m}-c_j\bigr).
\]
\end{itemize}

\textbf{Acknowledgments.}  
The authors thank Shengtong Zhang for discussions, including the reduction of
\(A_N\) to a cyclic tridiagonal matrix.  The second author was supported by
NSF grant DMS-2136217.

The authors used ChatGPT (OpenAI) to assist
with finding the chiral gauge transform for the real symmetric tridiagonal form as in \cite{JK19}, with locating
results on limiting eigenvalue distributions of real symmetric
tridiagonal matrices, and with editing.
The authors
independently verified all mathematical arguments and references and take full
responsibility for the contents of the manuscript.

\section{Two-line spectra}
\label{sec:unit-modulus}

The matrix \(A_N(\varphi,\vartheta)\) is in general neither self-adjoint nor normal, and
its diagonal and off-diagonal parts do not commute.  Both of the finite-\(N\)
theorems above rest on one structural observation: in a suitable orthonormal
basis, \(A_N(\varphi,\vartheta)\) is a \emph{cyclic tridiagonal matrix with
vanishing diagonal}, which is sometimes called the \emph{zero-diagonal form} \cite{JK19}.  Throughout the manuscript, we call an edge product of such a matrix the product of the $(k,k+1)$ and $(k+1,k)$ entry. For such a matrix the spectrum is controlled by the
products of opposite off-diagonal entries, and the geometry of the
spectrum is dictated by the signs of those products.  We first derive the basis and the products \(\rho_j\) in
\Cref{sub:basis}, then collect the needed tridiagonal linear algebra in
\Cref{sub:edge-products}, and finally apply it in
\Cref{sub:aligned,sub:arbitrary}.

\subsection{A change of basis}
\label{sub:basis}

For arbitrary \(N\ge3\), let $\varepsilon_N:=N\bmod2\in\{0,1\},
 \delta_N:=\frac{1-\varepsilon_N}{2},$ and $\omega:=\e^{2\pi i/N},$
and define
\begin{equation}
\label{eq:parity-chirp-basis}
 \psi_k(j)=\frac1{\sqrt N}
 \exp\!\left(\frac{\pi i j(j-\varepsilon_N)}N\right)\omega^{jk},
 \qquad j,k\in\Z/N\Z.
\end{equation}
The factor is \(N\)-periodic because \(N-\varepsilon_N\) is even;
thus \eqref{eq:parity-chirp-basis} is well defined for both parities of
\(N\), and the basis is orthonormal.  When \(N\) is even this is the basis
\(\exp(\pi i j^2/N)\omega^{jk}/\sqrt N\) used below in the even-dimensional argument, while for odd \(N\) the periodic exponential is
\(\exp(\pi i j(j-1)/N)\).

Multiplication by the shifted cosine gives
\[
 D_N(\vartheta)\psi_k
 =\frac{\e^{i\vartheta}\psi_{k+1}
 +\e^{-i\vartheta}\psi_{k-1}}2.
\]
A direct shift of the exponential yields
\[
 C_N\psi_k
 =\frac{\omega^{k+\delta_N}\psi_{k+1}
 +\omega^{1-\delta_N-k}\psi_{k-1}}2.
\]
We define the basis coefficients as 
\[
 \alpha_{k,N}:=\frac{\e^{i\vartheta}
 +\e^{i\varphi}\omega^{k+\delta_N}}2,
 \qquad
 \beta_{k,N}:=\frac{\e^{-i\vartheta}
 +\e^{i\varphi}\omega^{1-\delta_N-k}}2.
\]
Consequently,
\begin{equation}
\label{eq:tridiagonal-form}
 A_N(\varphi,\vartheta)\psi_k
 =\alpha_{k,N}\psi_{k+1}+\beta_{k,N}\psi_{k-1}.
\end{equation}
Thus \(A_N(\varphi,\vartheta)\) has zero diagonal in this basis and only
nearest-neighbour entries on the cycle.  The product of the two entries across the edge \(k\leftrightarrow k+1\) is
\begin{align}
 \rho_{k,N}(\varphi,\vartheta)
 &=\alpha_{k,N}\beta_{k+1,N}=\frac{\e^{i\varphi}}2
 \left(\cos\varphi+
 \cos\!\left(\frac{2\pi(k+\delta_N)}N-\vartheta\right)\right).
 \label{eq:general-phase-product}
\end{align}
This proves \eqref{eq:general-phase-product-intro} for every \(N\).  All products \(\rho_{j,N}\) lie on \(\e^{i\varphi}\R\), but this common phase alone does not force the eigenvalues onto two lines; the ordering of the products and the reflection reduction also enter.

\subsection{A tridiagonal matrix calculation}
\label{sub:edge-products}

Let \(T \in \mathbb C^{n \times n}\) be tridiagonal with zero diagonal.  Write
\(u_j:=T_{j,j+1}\), \(\ell_j:=T_{j+1,j}\), and
\(\rho_j:=u_j\ell_j\).  Thus
\begin{equation}
\label{eq:path-matrix}
 T=\begin{pmatrix}
 0&u_1\\
 \ell_1&0&u_2\\
 &\ell_2&0&\ddots\\
 &&\ddots&\ddots&u_{n-1}\\
 &&&\ell_{n-1}&0
 \end{pmatrix},
 \qquad \rho_j=u_j\ell_j.
\end{equation}
Let \(T_k\) be the leading \(k\times k\) block and put
\begin{equation*}
 p_k(z)=\det(zI_k-T_k),
 \qquad p_0(z)=1,
 \qquad p_1(z)=z.
\end{equation*}
Expanding \(p_{k+1}\) along its last row and then along the last column of the
remaining minor gives
\begin{align*}
 p_{k+1}(z)
 &=z\det(zI_k-T_k)-(-\ell_k)(-u_k)\det(zI_{k-1}-T_{k-1})=zp_k(z)-\rho_kp_{k-1}(z).
\end{align*}
Thus $p_0=1, p_1=z,$ and
\begin{equation}
\label{eq:continuant}
 p_{k+1}=zp_k-\rho_kp_{k-1},
\end{equation}
so that the characteristic polynomial \(\det(zI-T)=p_n(z)\) depends on the two
off-diagonal entries associated with an edge only through their product \(\rho_j\).  This is
why \eqref{eq:general-phase-product}, and not the individual coefficients
\(\alpha_{k,N},\beta_{k,N}\), is the relevant input.

A real symmetric tridiagonal matrix whose off-diagonal entries are all
nonzero has simple spectrum.  Indeed, if the first coordinate of an
eigenvector vanishes, the first row forces the second to vanish and the
recurrence then forces the entire vector to vanish.  For a fixed
eigenvalue the recurrence therefore determines every coordinate from the
first, so the eigenspace is at most
one-dimensional; symmetry then implies algebraic simplicity.  See, for
example, \cite[Section~1.2]{Teschl2000} or \cite[Chapter~I]{GantmacherKrein2002}.
The first step is that \(T\) can be brought into symmetric form.

\begin{lemma}\label{lem:diagonal-similarity}
Assume \(\rho_j\ne0\) for \(1\le j<n\).  For arbitrary square roots
\(w_j^2=\rho_j\), there is a diagonal matrix \(S\) such that
\begin{equation}
\label{eq:symmetrization}
 S^{-1}TS=
 \begin{pmatrix}
 0&w_1\\
 w_1&0&w_2\\
 &w_2&0&\ddots\\
 &&\ddots&\ddots&w_{n-1}\\
 &&&w_{n-1}&0
 \end{pmatrix},
 \qquad
 \frac{S_{j+1,j+1}}{S_{j,j}}=\frac{w_j}{u_j}.
\end{equation}
In particular, if \(\rho_j=\e^{2i\theta}a_j\) with \(\theta\in\R\) and \(a_j>0\), then
\begin{equation*}
 S^{-1}TS=\e^{i\theta}J,
 \qquad J=J^T\in\R^{n\times n},
 \qquad J_{j,j+1}=\sqrt{a_j}>0.
\end{equation*}
Hence \(\spec(T)\subset\e^{i\theta}\R\), and the eigenvalues are simple.
\end{lemma}

\begin{proof}
Write \(S=\diag(s_1,\ldots,s_n)\) with \(s_1=1\).  The two entries across
edge \(j\leftrightarrow j+1\) of \(S^{-1}TS\) are
\begin{equation*}
 (S^{-1}TS)_{j,j+1}=s_j^{-1}u_js_{j+1},
 \qquad
 (S^{-1}TS)_{j+1,j}=s_{j+1}^{-1}\ell_js_j.
\end{equation*}
Choose the diagonal entries recursively by $s_{j+1}=s_j\frac{w_j}{u_j}$ for $ 1\le j<n.$
Then the upper entry equals \(w_j\), while the lower entry is $s_{j+1}^{-1}\ell_js_j
 =\frac{u_j}{w_j}\ell_j
 =\frac{\rho_j}{w_j}
 =w_j.$
This proves \eqref{eq:symmetrization}.  If
\(\rho_j=\e^{2i\theta}a_j\) with \(a_j>0\), take $w_j=\e^{i\theta}\sqrt{a_j}.$
The resulting symmetric matrix is \(\e^{i\theta}J\), with \(J\) real and symmetric and with every off-diagonal entry nonzero.  Hence
\begin{equation*}
 \spec(T)=\e^{i\theta}\spec(J)\subset\e^{i\theta}\R.
\end{equation*}
In particular, since \(J\) is real symmetric,
it is diagonalizable, and because every off-diagonal entry is nonzero, every eigenvalue is simple; compare
\cite[Section~1.2]{Teschl2000}.
\end{proof}

\Cref{lem:diagonal-similarity} says that products \(\rho_j\) with a common phase
\(\e^{2i\theta}\) and \emph{positive} real factors force the spectrum onto the
single line \(\e^{i\theta}\R\).  The matrices of \Cref{sub:basis} need not
satisfy this hypothesis: by \eqref{eq:general-phase-product} the real factors
\(\cos\varphi+\cos(\cdot)\) may change sign.  When they do, a sign change is exactly the passage
from \(\rho\in\e^{2i\theta}(0,\infty)\) to
\(\rho\in\e^{2i(\theta+\pi/2)}(0,\infty)\), so one expects a second line
orthogonal to the first.  The mechanism that produces it is the following
splitting into the matrices \(B^TB\) and \(BB^T\), which we use three more times below.

\begin{proposition}\label{prop:odd-even}
Let \(T\) be as in \eqref{eq:path-matrix} with
\(\rho_j\ne0\) for \(1\le j<n\), let \(w_j\) be
square roots \(w_j^2=\rho_j\), and let \(S\) be as in
\Cref{lem:diagonal-similarity}.  Let \(P\) be the permutation listing the odd
indices before the even ones,
\[
 P: (1,2,\ldots,n)\longmapsto
 (1,3,5,\ldots,2,4,6,\ldots),
\]
and set \(s=\lfloor n/2\rfloor\).  Define
\(B\in\C^{(n-s)\times s}\) to be the bidiagonal matrix with entries
\[
 B_{r,r}:=w_{2r-1}\quad(1\le r\le s),
 \qquad
 B_{r+1,r}:=w_{2r}
 \quad(1\le r\le s,\ r+1\le n-s).
\]
Here \(B^T\) denotes ordinary transpose, not conjugate transpose.  With
the convention \(w_0=w_n=0\), we have
\begin{equation}
\label{eq:block-B}
 P^TS^{-1}TSP=
 \begin{pmatrix}0&B\\B^T&0\end{pmatrix}.
\end{equation}
Moreover,
\begin{equation}
\label{eq:Schur}
 \det(zI-T)=z^{\,n-2s}\det(z^2I_s-B^TB),
 \qquad
 \det(tI_{n-s}-BB^T)=t^{\,n-2s}\det(tI_s-B^TB),
\end{equation}
and the two matrices \(B^TB\) and \(BB^T\) are the tridiagonal matrices
\begin{equation}
\label{eq:gram-entries}
\begin{aligned}
 (B^TB)_{r,r}&=w_{2r-1}^2+w_{2r}^2,
 &(B^TB)_{r,r+1}&=w_{2r}w_{2r+1},\\
 (BB^T)_{r,r}&=w_{2r-2}^2+w_{2r-1}^2,
 &(BB^T)_{r,r+1}&=w_{2r-1}w_{2r}.
\end{aligned}
\end{equation}
The formulas are understood whenever the displayed indices lie in range.
The off-diagonal entries of \(B^TB\) are the consecutive
products \(w_2w_3,w_4w_5,\ldots\), and those of \(BB^T\) are
\(w_1w_2,w_3w_4,\ldots\); every consecutive product \(w_jw_{j+1}\) occurs in
exactly one of the two matrices \(B^TB\) and \(BB^T\), for
\(1\le j\le n-2\).
\end{proposition}

\begin{proof}
By \Cref{lem:diagonal-similarity}, \(S^{-1}TS\) is symmetric tridiagonal with
zero diagonal and off-diagonal entries \(w_1,\ldots,w_{n-1}\).  A zero
diagonal means that this matrix maps the span of the odd basis vectors into
the span of the even ones and conversely, which is exactly
\eqref{eq:block-B}.  The odd index \(2r-1\) is adjacent only to the even
indices \(2r-2\) and \(2r\), with weights \(w_{2r-2}\) and \(w_{2r-1}\); this
gives the stated entries of \(B\), and \eqref{eq:gram-entries} follows by
direct multiplication.

For \(z\ne0\), a direct block determinant calculation gives
\begin{align*}
 \det\begin{pmatrix}zI_{n-s}&-B\\-B^T&zI_s\end{pmatrix}
 &=z^{n-s}\det\left(zI_s-z^{-1}B^TB\right)
 =z^{\,n-2s}\det(z^2I_s-B^TB).
\end{align*}
Since both sides are polynomials in \(z\), the identity extends from
\(z\ne0\) to every \(z\in\C\).  Applying the same Schur-complement
calculation with the two block sizes interchanged, initially for
\(t\ne0\), and then using polynomial continuation proves the second
identity in \eqref{eq:Schur}.
\end{proof}
We then obtain the first result showing inclusion of the spectrum in the two perpendicular lines. 
\begin{theorem}
\label{thm:one-sign-change}
Let \(T\) be a zero-diagonal tridiagonal matrix as in
\eqref{eq:path-matrix}, and suppose that, for some \(\alpha\in\R\) and real
numbers \(c_j\), its products \(\rho_j\) satisfy $\rho_j=\e^{2i\alpha}c_j.$
Assume that, on every maximal consecutive block where \(c_j\ne0\), the
signs change at most once.  Then
\[
 \spec(T)\subset\cX_\alpha.
\]
\end{theorem}

\begin{proof}
Since the continuant recurrence \eqref{eq:continuant} depends on the entries
\(u_j,\ell_j\) only through the products
\(\rho_j=u_j\ell_j\), we may replace \(u_j=\ell_j=0\) whenever
\(\rho_j=0\). This does not change the characteristic polynomial and splits
the path into components on which all \(\rho_j\) are nonzero. It is therefore
enough to treat one such component.

If the \(c_j\) have constant sign, then \Cref{lem:diagonal-similarity} gives
\[
\operatorname{spec}(T)\subset e^{i\alpha}\mathbb{R}
\qquad\text{or}\qquad
\operatorname{spec}(T)\subset e^{i(\alpha+\pi/2)}\mathbb{R},
\]
according to that sign. Hence only the case of exactly one sign change
remains.

By \Cref{lem:diagonal-similarity}, we may take \(T\) symmetric, with
off-diagonal entries \(w_j\) satisfying
\[
w_j^2=\rho_j.
\]
Thus the \(w_j\) have phase \(e^{i\alpha}\) on one side of the sign change
and \(i e^{i\alpha}\) on the other. Apply the odd--even decomposition of
\Cref{prop:odd-even}. In the two Gram matrices \(B^T B\) and \(B B^T\), the
diagonal entries are sums of terms
\[
w_j^2=\rho_j\in e^{2i\alpha}\mathbb{R},
\]
while the off-diagonal entries are products \(w_jw_{j+1}\). All such products
belong to \(e^{2i\alpha}\mathbb{R}\), except the one crossing the unique sign
change. By \Cref{prop:odd-even}, this exceptional product occurs in exactly
one of \(B^T B\) and \(B B^T\). If \(K\) denotes the other matrix, then
\(e^{-2i\alpha}K\) is real symmetric. Hence
\[
\operatorname{spec}(K)\subset e^{2i\alpha}\mathbb{R}.
\]

By \eqref{eq:Schur}, the nonzero spectra of \(B^T B\) and \(B B^T\) coincide.
Therefore every nonzero \(z\in\operatorname{spec}(T)\) satisfies
\[
z^2\in\operatorname{spec}(K)\subset e^{2i\alpha}\mathbb{R}.
\]
Equivalently,
\[
z\in e^{i\alpha}\mathbb{R}
   \cup e^{i(\alpha+\pi/2)}\mathbb{R}
   =\mathcal{X}_\alpha.
\]
The conclusion is immediate for \(z=0\).
\end{proof}
\subsection{Reflection symmetry when
\texorpdfstring{\(\vartheta\in2\pi\Z/N\)}{the phase is aligned}}
\label{sub:aligned}

The model \eqref{eq:tridiagonal-form} lives on the cycle
\(\Z/N\Z\), whereas \Cref{thm:one-sign-change} is a statement about paths.
The passage from one to the other is the standard symmetry reduction:

Let \(M\ge2\), and let $L$ be a zero-diagonal cyclic tridiagonal
matrix of dimension $2M$, indexed by $\mathbb Z/(2M)\mathbb Z$, and write
\[
 u_k=L_{k,k+1},\qquad \ell_k=L_{k+1,k},\qquad \rho_k=u_k\ell_k.
\]
Let $\mathcal R$ be the reflection of the cycle defined by
$\mathcal Re_k=e_{-k}$; it fixes the two indices \(0\) and \(M\).  Suppose
$\mathcal RL=L\mathcal R$.  In particular,
\[
 u_k=\ell_{-k-1},\qquad
 \ell_k=u_{-k-1},\qquad
 \rho_k=\rho_{-k-1}.
\]
Since
\(\mathcal R^2=I\), the space splits into the \(\pm1\) eigenspaces of
\(\mathcal R\), and \(L\) preserves each.  The \(+1\) and \(-1\) reflection subspaces have orthonormal bases
\[
 p_0=e_0,\qquad
 p_k=\frac{e_k+e_{-k}}{\sqrt2}\ (1\le k<M),\qquad
 p_M=e_M,
\]
and
\[
 m_k=\frac{e_k-e_{-k}}{\sqrt2}\ (1\le k<M),
\]
respectively. Hence \(\dim L_+=M+1\) and \(\dim L_-=M-1\).

Folding the cycle preserves each interior edge product \(\rho_k\).
At a fixed endpoint, however, the two reflected couplings add. For example,
the entries of \(L_+\) between \(p_0\) and \(p_1\) are
\(\sqrt2\,u_0\) and \(\sqrt2\,\ell_0\), whose product is \(2\rho_0\).
The same calculation at \(M\) gives \(2\rho_{M-1}\). Thus
\[
 L_+:\quad
 2\rho_0,\rho_1,\ldots,\rho_{M-2},2\rho_{M-1}.
\]
The antisymmetric subspace contains neither fixed vertex, so its endpoint
edges disappear and
\[
 L_-:\quad \rho_1,\ldots,\rho_{M-2}.
\]
Since the extra factors \(2\) are positive, they do not change the signs of
the edge products.

\begin{proof}[Proof of \Cref{thm:intro-unit-modulus}]
We first consider \(\vartheta=0\). Write \(N=2M\) and let
\[
L=\widehat A_N(\varphi,0)
\]
be the matrix of \(A_N(\varphi,0)\) in the basis
\eqref{eq:parity-chirp-basis}. Since \(N\) is even,
\(\delta_N=1/2\), and \eqref{eq:tridiagonal-form} gives
\[
Le_k=\alpha_{k,N}e_{k+1}+\beta_{k,N}e_{k-1},
\]
where
\[
\alpha_{k,N}
=\frac{1+\e^{i\varphi}\omega^{k+1/2}}{2},
\qquad
\beta_{k,N}
=\frac{1+\e^{i\varphi}\omega^{1/2-k}}{2}.
\]
Hence
\[
\alpha_{-k,N}=\beta_{k,N},
\qquad
\beta_{-k,N}=\alpha_{k,N}.
\]

Let \(\mathcal R e_k=e_{-k}\). Then
\[
\mathcal RLe_k
=\alpha_{k,N}e_{-k-1}
+\beta_{k,N}e_{-k+1},
\]
whereas
\[
L\mathcal Re_k
=Le_{-k}
=\alpha_{-k,N}e_{-k+1}
+\beta_{-k,N}e_{-k-1}.
\]
Thus \(\mathcal RL=L\mathcal R\). The matrix \(L\) therefore preserves the
\(+1\) and \(-1\) eigenspaces of \(\mathcal R\), and in the
reflection-adapted basis it decomposes as
\[
L=L_+\oplus L_-.
\]
By \eqref{eq:fold-plus-products}--\eqref{eq:fold-minus-products},
the corresponding edge-product lists are
\begin{align}
 L_+&:\quad 2\rho_0,\rho_1,\ldots,\rho_{M-2},2\rho_{M-1},
 \label{eq:fold-plus-products}\\
 L_-&:\quad \rho_1,\ldots,\rho_{M-2}.
 \label{eq:fold-minus-products}
\end{align}

We now determine the signs of these products. Since
\(\delta_N=1/2\) and \(\vartheta=0\),
\eqref{eq:general-phase-product} gives
\[
\rho_k
=\frac{\e^{i\varphi}}{2}
\left(
\cos\varphi+
\cos\frac{(2k+1)\pi}{N}
\right)
=\e^{i\varphi}c_{k,N}^{(\varphi)},
\]
where
\[
c_{k,N}^{(\varphi)}
:=\frac12\left(
\cos\varphi+
\cos\frac{(2k+1)\pi}{N}
\right).
\]
For \(0\le k\le M-1\), the angles $\frac{(2k+1)\pi}{N}$
increase strictly from \(\pi/N\) to \(\pi-\pi/N\). Since
\(\cos t\) is strictly decreasing on \((0,\pi)\), the sequence
\[
c_{0,N}^{(\varphi)},\ldots,c_{M-1,N}^{(\varphi)}
\]
is strictly decreasing and therefore changes sign at most once.

Multiplying the first and last products by \(2\) does not change their
signs, and passing to the sublist
\(\rho_1,\ldots,\rho_{M-2}\) cannot create an additional sign change.
Hence both \(L_+\) and \(L_-\) satisfy the hypotheses of
\Cref{thm:one-sign-change}. Since
\[
\rho_k
=\e^{i\varphi}c_{k,N}^{(\varphi)}
=\e^{2i(\varphi/2)}c_{k,N}^{(\varphi)},
\]
we apply that theorem with \(\alpha=\varphi/2\) and obtain
\[
\spec(L_+),\spec(L_-)\subset\cX_{\varphi/2}.
\]
Therefore
\[
\spec A_N(\varphi,0)
=\spec(L)
=\spec(L_+)\cup\spec(L_-)
\subset\cX_{\varphi/2}.
\]

Now let
\[
\vartheta=\frac{2\pi m}{N}.
\]
Since \(S e_j=e_{j+1}\),
\[
S^mD_N(\vartheta)S^{-m}=D_N(0),
\]
while \(C_N=(S+S^{-1})/2\) commutes with \(S^m\). Hence
\[
S^mA_N(\varphi,\vartheta)S^{-m}
=A_N(\varphi,0).
\]
Thus all matrices with
\(\vartheta\in2\pi\Z/N\) are unitarily equivalent, and
\eqref{eq:rotated-cross-intro} follows for every such \(\vartheta\).

It remains to prove the bounds along the two lines. By the preceding
unitary equivalence, it is enough to take \(\vartheta=0\). Let
\[
A_N(\varphi,0)x=zx,
\qquad
\|x\|=1,
\]
and set
\[
a:=x^*D_N(0)x,
\qquad
b:=x^*C_Nx.
\]
Both \(D_N(0)\) and \(C_N\) are Hermitian with norm at most \(1\), so
\[
a,b\in\R,
\qquad
|a|,|b|\le1.
\]
Taking the scalar product of the eigenvalue equation with \(x\) gives
\[
z=a+\e^{i\varphi}b.
\]
Thus
\[
\begin{aligned}
\e^{-i\varphi/2}z
&=\e^{-i\varphi/2}a+\e^{i\varphi/2}b=\cos(\varphi/2)(a+b)
+i\sin(\varphi/2)(b-a).
\end{aligned}
\]

Suppose first that
\[
z\in\e^{i\varphi/2}\R.
\]
Then \(\e^{-i\varphi/2}z\) is real. If
\(\sin(\varphi/2)\ne0\), its imaginary part gives \(a=b\), and hence
\[
|z|
=2|a|\,|\cos(\varphi/2)|
\le2|\cos(\varphi/2)|.
\]
If \(\sin(\varphi/2)=0\), the same estimate follows directly from
\[
|z|
=|\cos(\varphi/2)(a+b)|
\le2|\cos(\varphi/2)|.
\]

Suppose now that
\[
z\in\e^{i(\varphi/2+\pi/2)}\R.
\]
Then \(\e^{-i\varphi/2}z\) is purely imaginary. If
\(\cos(\varphi/2)\ne0\), its real part gives \(a=-b\), and therefore
\[
|z|
=2|b|\,|\sin(\varphi/2)|
\le2|\sin(\varphi/2)|.
\]
If \(\cos(\varphi/2)=0\), the same estimate follows directly from
\[
|z|
=|\sin(\varphi/2)(b-a)|
\le2|\sin(\varphi/2)|.
\]
This proves the stated bounds.
\end{proof}

\subsection{Arbitrary potential phase and boundary twist}
\label{sub:arbitrary}

We now derive the dependence of the characteristic polynomial on the potential
phase \(\vartheta\) and the boundary twist \(\tau\). This computation is very similar to how Chambers' formula is derived in the case of AMO \cite{Chambers1965}. Recall that
\[
 \chi_{N,\varphi,\vartheta}^{\tau}(z)
 :=\det\bigl(zI-A_N^\tau(\varphi,\vartheta)\bigr),
 \qquad
 \tau=\e^{i\kappa}.
\]
Set $\xi:=\e^{i\vartheta},$ since
\[
 \cos\!\left(\frac{2\pi j}{N}+\vartheta\right)
 =\frac12\left(\xi\omega^j+\xi^{-1}\omega^{-j}\right),
\]
the \(j\)-th diagonal entry of
\(zI-A_N^\tau(\varphi,\vartheta)\) is
\[
 z-\frac12\left(\xi\omega^j+\xi^{-1}\omega^{-j}\right).
\]
The off-diagonal entries are independent of \(\xi\). Hence
\(\chi_{N,\varphi,\vartheta}^{\tau}(z)\) is a Laurent polynomial in
\(\xi\), with powers between \(-N\) and \(N\).

Cyclic translation changes \(\vartheta\) to
\(\vartheta+2\pi/N\) without changing the characteristic polynomial.
Equivalently, the Laurent polynomial is invariant under
\[
 \xi\longmapsto\omega\xi.
\]
Thus a term \(a_r\xi^r\) can occur only if
\(\omega^r=1\), that is, only if \(N\mid r\). Since \(|r|\le N\),
the only possible powers are $\xi^{-N}, 1, \xi^N.$

We next compute the coefficients of \(\xi^{\pm N}\). To obtain the term
\(\xi^N\), one must choose $-\frac{\xi\omega^j}{2}$
from every diagonal entry in the determinant expansion. Therefore, we find for the coefficient
\[
 [\xi^N]\chi_{N,\varphi,\vartheta}^{\tau}(z)
 =
 \prod_{j=0}^{N-1}\left(-\frac{\omega^j}{2}\right).
\]
Since
\[
 \prod_{j=0}^{N-1}\omega^j
 =\omega^{N(N-1)/2}
 =(-1)^{N-1},
\]
we obtain
\[
 \prod_{j=0}^{N-1}\left(-\frac{\omega^j}{2}\right)
 =(-1)^N2^{-N}(-1)^{N-1}
 =-2^{-N}.
\]
The same calculation gives the coefficient \(-2^{-N}\) for
\(\xi^{-N}\). Hence the entire \(\vartheta\)-dependent contribution is
\[
 -2^{-N}(\xi^N+\xi^{-N})
 =-2^{1-N}\cos(N\vartheta).
\]

It remains to determine the dependence on the twist. The parameter \(\tau\)
occurs only in the two entries joining the endpoints of the cyclic chain.
Any determinant term using both of these entries contains the product
\(\tau\tau^{-1}=1\) and is therefore independent of \(\tau\).
A term containing exactly one of them must traverse the entire cycle, so
there are only two such terms, corresponding to the two orientations of
the cycle.

Each oriented cycle is an \(N\)-cycle and therefore has permutation sign
\((-1)^{N-1}\). Its \(N\) matrix entries are hopping terms
\(-\e^{i\varphi}/2\). Thus the two contributions are
\[
 (-1)^{N-1}
 \left(-\frac{\e^{i\varphi}}2\right)^N\tau
 =-2^{-N}\e^{iN\varphi}\tau
\text{ and }
 (-1)^{N-1}
 \left(-\frac{\e^{i\varphi}}2\right)^N\tau^{-1}
 =-2^{-N}\e^{iN\varphi}\tau^{-1}.
\]
Their sum is
\[
 -2^{-N}\e^{iN\varphi}(\tau+\tau^{-1})
 =-2^{1-N}\e^{iN\varphi}\cos\kappa.
\]

All remaining terms are independent of both \(\vartheta\) and \(\tau\).
Collecting them into a monic polynomial \(Q_{N,\varphi}(z)\), we obtain
\begin{equation}
 \label{eq:twisted-Chambers-Q}
 \chi_{N,\varphi,\vartheta}^{\tau}(z)
 =
 Q_{N,\varphi}(z)
 -2^{1-N}\cos(N\vartheta)
 -2^{1-N}\e^{iN\varphi}\cos\kappa.
\end{equation}
Evaluating this identity at \((\vartheta,\tau)=(0,1)\) gives
\[
 \chi_{N,\varphi,0}^{1}(z)
 =
 Q_{N,\varphi}(z)
 -2^{1-N}
 -2^{1-N}\e^{iN\varphi}.
\]
Subtracting the two formulas yields
\begin{equation}
 \label{eq:twisted-Chambers-intro2}
 \chi_{N,\varphi,\vartheta}^{\tau}(z)
 =
 \chi_{N,\varphi,0}^{1}(z)
 +2^{1-N}\bigl(1-\cos(N\vartheta)\bigr)
 +2^{1-N}\e^{iN\varphi}\bigl(1-\cos\kappa\bigr).
\end{equation}

For periodic boundary conditions, \(\tau=1\), so \(\kappa=0\), and the
last term vanishes. Therefore
\begin{equation}
 \label{eq:Chambers-phase}
 \chi_{N,\varphi,\vartheta}^{1}(z)
 =
 \chi_{N,\varphi,0}^{1}(z)
 +2^{1-N}\bigl(1-\cos(N\vartheta)\bigr).
\end{equation}
This is the frequency-\(1/N\) determinant formula in our normalization;
compare
\cite{Chambers1965,LamoureuxMingo2007,
JitomirskayaKonstantinovKrasovsky2022}.

Formula \eqref{eq:Chambers-phase} makes the potential-phase dependence
particularly simple. If $\vartheta\in\frac{2\pi}{N}\Z,$
then \(\cos(N\vartheta)=1\), and hence
\[
 \chi_{N,\varphi,\vartheta}^{1}(z)
 =
 \chi_{N,\varphi,0}^{1}(z).
\]
Thus all these periodic fibres are isospectral. For a general
\(\vartheta\), however, the phase changes the constant term of the
characteristic polynomial, and the spectrum need not remain on the two
distinguished lines.

Finally, we specialize the determinant identity to the fibre corresponding
to the Scottish flag matrix. Take
\[
 \varphi=\vartheta=\frac{\pi}{2},
 \qquad
 \tau=i^N=\e^{iN\pi/2},
\]
so we may choose $\kappa=\frac{N\pi}{2}.$
If \(4\mid N\), then $\cos(N\vartheta)=\cos\kappa=1.$
Both correction terms in \eqref{eq:twisted-Chambers-intro2} vanish, and
therefore $\chi_{N,\pi/2,\pi/2}^{i^N}(z)
 =
 \chi_{N,\pi/2,0}^{1}(z).$

If \(N\equiv2\pmod4\), then $\cos(N\vartheta)=\cos\kappa=-1$ and $
 \e^{iN\varphi}=-1.$
The two correction terms in \eqref{eq:twisted-Chambers-intro2} are therefore
\[
 2^{1-N}(1-(-1))
 \qquad\text{and}\qquad
 -2^{1-N}(1-(-1)),
\]
and cancel. Hence again
\[
 \chi_{N,\pi/2,\pi/2}^{i^N}(z)
 =
 \chi_{N,\pi/2,0}^{1}(z).
\]
This proves \eqref{eq:twist-cancellation-intro} for every even \(N\).

If \(N\) is odd, then $\cos\frac{N\pi}{2}=0,$
so $\cos(N\vartheta)=\cos\kappa=0.$
Equation \eqref{eq:twisted-Chambers-Q} therefore gives $\chi_{N,\pi/2,\pi/2}^{i^N}(z)
 =
 Q_{N,\pi/2}(z),$
which is \eqref{eq:odd-central-Chambers-intro}. Together with
\eqref{eq:general-phase-product}, this completes the proof of
\Cref{thm:intro-arbitrary-phase}.

\subsection{Phase opening of the central polynomial}
\label{sub:phase-opening}

Because the central polynomial is independent of the potential phase, we may choose the phase so that one edge product vanishes.  The cyclic matrix then opens into a path, and the sign-change result from \Cref{thm:one-sign-change} applies directly.

\begin{lemma}
\label{lem:phase-opening}
Set
\[
 \vartheta_*:=\varphi+\frac{2\pi\delta_N}{N}-\pi,
 \qquad
 \delta_N=\frac{1-\varepsilon_N}{2}.
\]
For the edge products in \eqref{eq:general-phase-product}, one has
\begin{equation}
\label{eq:opened-edge-products}
 \rho_{0,N}(\varphi,\vartheta_*)=0,
 \qquad
 \rho_{j,N}(\varphi,\vartheta_*)
 =\e^{i\varphi}\sin\frac{\pi j}{N}
   \sin\left(\frac{\pi j}{N}-\varphi\right),
 \quad 1\le j<N.
\end{equation}
Moreover, \(Q_{N,\varphi}(z)\) is the characteristic polynomial of a
zero-diagonal tridiagonal path whose consecutive edge products are the
\(N-1\) numbers in \eqref{eq:opened-edge-products} with
\(1\le j<N\).  Zero products are allowed.
\end{lemma}

\begin{proof}
Insert \(\vartheta_*\) into \eqref{eq:general-phase-product}.  Since
\[
 \frac{2\pi(k+\delta_N)}N-\vartheta_*
 =\frac{2\pi k}{N}-\varphi+\pi,
\]
we obtain
\begin{align*}
 \rho_{k,N}(\varphi,\vartheta_*)
 &=\frac{\e^{i\varphi}}2
   \left(\cos\varphi-\cos\left(\frac{2\pi k}{N}-\varphi\right)\right)=\e^{i\varphi}\sin\frac{\pi k}{N}
   \sin\left(\frac{\pi k}{N}-\varphi\right).
\end{align*}
This gives \eqref{eq:opened-edge-products}, including the zero product at
\(k=0\).

It remains to identify the resulting path polynomial with
\(Q_{N,\varphi}\).  Write
\(L=\widehat A_N(\varphi,\vartheta_*)\), with the coefficients
\(\alpha_{k,N},\beta_{k,N}\) from \Cref{sub:basis}.  Expand \(\det(zI-L)\) by permutations.  Since \(L\) has zero diagonal and
only nearest-neighbour entries, a nonzero permutation consists either of
fixed points and disjoint transpositions \(k\leftrightarrow k+1\), or of one
of the two permutations that traverse the whole cycle.  A fixed point
contributes \(z\), while a transposition across the edge
\(k\leftrightarrow k+1\) contributes
\[
 -\alpha_{k,N}\beta_{k+1,N}=-\rho_{k,N}.
\]
Since \(\rho_{0,N}=0\), every term containing the transposition
\(0\leftrightarrow1\) vanishes.  The remaining fixed-point and
transposition terms are therefore exactly the determinant expansion of the
path obtained by deleting this edge, whose vertices may be ordered as
\[
 1,2,\ldots,N-1,0
\]
and whose edge products are
\(\rho_{1,N},\ldots,\rho_{N-1,N}\).  If \(H(z)\) denotes its
characteristic polynomial, the only additional terms in the cyclic
determinant are the two oriented cycles.  Hence
\[
 \det(zI-L)
 =
 H(z)
 -\prod_{k=0}^{N-1}\alpha_{k,N}
 -\prod_{k=0}^{N-1}\beta_{k,N}.
\]

Using \(\prod_{k=0}^{N-1}(x+y\omega^k)=x^N-(-y)^N\), one finds, for any
real \(\vartheta\),
\[
 \prod_{k=0}^{N-1}\alpha_{k,N}
 =2^{-N}\bigl(\e^{iN\vartheta}+\e^{iN\varphi}\bigr),
 \qquad
 \prod_{k=0}^{N-1}\beta_{k,N}
 =2^{-N}\bigl(\e^{-iN\vartheta}+\e^{iN\varphi}\bigr).
\]
At \(\vartheta=\vartheta_*\), their sum is
\[
 2^{1-N}\cos(N\vartheta_*)+2^{1-N}\e^{iN\varphi}.
\]
On the other hand, \eqref{eq:twisted-Chambers-Q} with \(\tau=1\) gives
\[
 \det(zI-L)
 =Q_{N,\varphi}(z)
 -2^{1-N}\cos(N\vartheta_*)
 -2^{1-N}\e^{iN\varphi}.
\]
Comparing this identity with our formula above yields
\(H(z)=Q_{N,\varphi}(z)\).  This proves the lemma.
\end{proof}

\begin{proof}[Proof of \Cref{thm:intro-central-cross}]
By \Cref{lem:phase-opening}, \(Q_{N,\varphi}\) is the characteristic
polynomial of a zero-diagonal path with edge products
\[
 \e^{i\varphi}\sin\frac{\pi j}{N}
 \sin\left(\frac{\pi j}{N}-\varphi\right),
 \qquad 1\le j<N.
\]
The first sine is positive.  As \(j\) increases from \(1\) to \(N-1\), the
arguments \(\pi j/N-\varphi\) fill an interval of length strictly less than
\(\pi\).  Since consecutive zeros of the sine function are \(\pi\) apart,
the second sine changes sign at most once after zero terms are removed.
Therefore \Cref{thm:one-sign-change}, with \(\alpha=\varphi/2\), gives
\begin{equation*}
 Q_{N,\varphi}(z)=0
 \quad\Longrightarrow\quad
 z\in\cX_{\varphi/2}.
\end{equation*}

A zero-diagonal path of size \(N\) has an even characteristic polynomial
when \(N\) is even and an odd one when \(N\) is odd.  Hence
\[
 Q_{N,\varphi}(z)=z^{\varepsilon_N}P_{N,\varphi}(z^2)
\]
for a monic polynomial \(P_{N,\varphi}\) of degree \(\ell_N\).  Define
\[
 R_{N,\varphi}(t)
 :=\e^{-i\ell_N\varphi}
 P_{N,\varphi}(\e^{i\varphi}t).
\]
If \(z\) is a zero of \(Q_{N,\varphi}\), then
\(t=\e^{-i\varphi}z^2\in\R\) by the two-line containment just proved.
Thus all zeros of the monic polynomial \(R_{N,\varphi}\) are real, so
\(R_{N,\varphi}\in\R[t]\).  This proves
\eqref{eq:central-real-rooted-form}.

If \(\tau=\e^{i\kappa}\) satisfies
\eqref{eq:central-fibre-condition}, then
\eqref{eq:twisted-Chambers-Q} immediately gives
\[
 \chi_{N,\varphi,\vartheta}^{\tau}(z)=Q_{N,\varphi}(z),
\]
so the same two-line containment holds for that fibre.  In particular, the
choice
\[
 \vartheta=\frac{\pi}{2N},\qquad \tau=i
\]
always satisfies the condition, and therefore
\[
 Q_{N,\varphi}(z)
 =\det\!\left(
 zI-A_N^i\!\left(\varphi,\frac{\pi}{2N}\right)
 \right).
\]

It remains only to bound the roots along the two lines.  Let \(x\) be a
normalized eigenvector of this matrix with eigenvalue \(z\), and set
\[
 a:=x^*D_N\!\left(\frac{\pi}{2N}\right)x,
 \qquad
 b:=x^*\frac{S_i+S_i^{-1}}2x.
\]
Both \(a\) and \(b\) are real and lie in \([-1,1]\).  The eigenvalue
identity gives
\[
 \e^{-i\varphi/2}z
 =\cos\frac{\varphi}{2}(a+b)
 +i\sin\frac{\varphi}{2}(b-a).
\]
If \(z\in\e^{i\varphi/2}\R\), the left side is real.  When
\(\sin(\varphi/2)\ne0\), this forces \(a=b\), and hence
\(|z|\le2|\cos(\varphi/2)|\); when \(\sin(\varphi/2)=0\), the same bound
follows directly from \(|a+b|\le2\).  Similarly, if
\(z\in\e^{i(\varphi/2+\pi/2)}\R\), then the left side is purely imaginary.
This gives \(a=-b\) when \(\cos(\varphi/2)\ne0\), and otherwise the bound is
direct; in both cases
\[
 |z|\le2\left|\sin\frac{\varphi}{2}\right|.
\]
Finally, \(t=\e^{-i\varphi}z^2\) is nonnegative on the first line and
nonpositive on the second.  Squaring the two bounds gives the stated
interval for the zeros of \(R_{N,\varphi}\).
\end{proof}

\section{The \texorpdfstring{periodic case \(\varphi=\pi/2\)}{periodic pi/2 case} and the Scottish flag matrix}
\label{sec:scottish}

At \(\varphi=\pi/2\), we have that
\eqref{eq:general-phase-product} satisfy, for $N=2M$ for some positive integer $M$,
\begin{equation}
\label{eq:scottish-products}
 \rho_{k,N}=\frac i2\cos\!\left(\frac{(2k+1)\pi}{N}\right),
 \qquad
 \rho_{M-1-k,N}=-\rho_{k,N},
\end{equation}
because
\(\frac{(2(M-1-k)+1)\pi}{N}
=\pi-\frac{(2k+1)\pi}{N}\).

Throughout this section, \(J_\ell(0)\) denotes the \(\ell\times\ell\) nilpotent
Jordan block, and for an \(N\times N\) matrix \(X\) we set
\(\cG_0(X):=\ker X^N\).

\subsection{Reversal identity for products
\texorpdfstring{\(\rho_j\)}{rho-j}}
\label{sub:mirrored}

\begin{theorem}\label{thm:mirrored-path}
Let \(n=2s+1\) and let \(T\) be a zero-diagonal tridiagonal matrix as in
\eqref{eq:path-matrix} whose products \(\rho_j\) satisfy the following reversal identity: For some
\(\theta\in\R\) and positive numbers \(a_1,\ldots,a_s\), reversing their
order flips every sign:
\begin{equation}
\label{eq:mirrored-rho}
 \rho_j=\e^{2i\theta}a_j,
 \qquad
 \rho_{2s+1-j}=-\e^{2i\theta}a_j,
 \qquad 1\le j\le s.
\end{equation}
Let \(K\) be the matrix of \Cref{prop:odd-even} whose off-diagonal
entries avoid the sign change, that is,
\begin{equation*}
 K=\begin{cases}
 B^TB,&s\ \text{odd},\\
 BB^T,&s\ \text{even},
 \end{cases}
\end{equation*}
and write
\begin{equation}
\label{eq:K-H}
 K=\e^{2i\theta}H,
 \qquad H=H^T\in\R^{d\times d},
 \qquad d=\begin{cases}s,&s\ \text{odd},\\s+1,&s\ \text{even}.
 \end{cases}
\end{equation}
Then \(H\) is a real symmetric tridiagonal matrix and every off-diagonal entry is nonzero.  Its
dimension \(d\) is odd; write \(d=2m+1\).  Its spectrum is simple and has the
form
\begin{equation*}
 \spec(H)=\{0,\pm\nu_1,\ldots,\pm\nu_m\},
 \qquad 0<\nu_1<\cdots<\nu_m.
\end{equation*}
If \(s=2m\), then
\begin{equation}
\label{eq:mirror-factor-even}
 \det(zI-T)=z\prod_{r=1}^{m}(z^4-\e^{4i\theta}\nu_r^2).
\end{equation}
If \(s=2m+1\), then
\begin{equation}
\label{eq:mirror-factor-odd}
 \det(zI-T)=z^3\prod_{r=1}^{m}(z^4-\e^{4i\theta}\nu_r^2).
\end{equation}
Consequently,
\begin{equation}
\label{eq:mirror-cross-conclusion}
 \spec(T)\subset\cX_\theta.
\end{equation}
The Jordan block at zero is \(J_1(0)\) in
\eqref{eq:mirror-factor-even} and \(J_3(0)\) in
\eqref{eq:mirror-factor-odd}.
\end{theorem}

The proof is given in \Cref{app:mirrored}.  

\begin{corollary}
\label{cor:cyclic-cross}
Let \(M\geq 2\), and let \(L\) be a zero-diagonal cyclic tridiagonal
matrix of dimension \(2M\), indexed by \(\mathbb Z/(2M)\mathbb Z\).
Assume that \(L\) commutes with the reflection $\mathcal Re_k=e_{-k}.$
Write $\rho_k=L_{k,k+1}L_{k+1,k}$
for the product associated with the edge \(k\leftrightarrow k+1\).

Consider the \(M\) edges on the half-cycle from the fixed vertex \(0\)
to the fixed vertex \(M\), namely
\[
    \rho_0,\rho_1,\ldots,\rho_{M-1}.
\]
Suppose that, for some \(\theta\in\mathbb R\),
\[
    \rho_k=e^{2i\theta}c_k,
    \qquad 0\leq k\leq M-1,
\]
where \(c_k\in\mathbb R\) satisfy
\[
    c_{M-1-k}=-c_k
\]
and
\[
    c_k>0
    \qquad\text{for }0\leq k<\frac{M-1}{2}.
\]
If \(M\) is odd, the reversal symmetry forces
\[
    c_{(M-1)/2}=0;
\]
apart from this forced zero, assume that all \(c_k\) are nonzero.
Then
\[
    \spec(L)\subset\cX_\theta.
\]
\end{corollary}

\begin{proof}
The reversal \(k\mapsto M-1-k\) acts on the \(M\) edge indices of the
half-cycle. Hence the assumptions imply that the sequence
\[
    c_0,c_1,\ldots,c_{M-1}
\]
is positive on its first half and negative on its second half. If
\(M\) is odd, its middle term is zero. Thus, after removing this
possible zero, the signs change at most once.

Since \(L\) commutes with \(\mathcal R\), the reflection decomposition from
Section~2.3 gives
\[
    L=L_+\oplus L_-.
\]
The edge-product lists of the two tridiagonal path blocks are
\[
    L_+:\quad
    2\rho_0,\rho_1,\ldots,\rho_{M-2},2\rho_{M-1},
\]
and
\[
    L_-:\quad
    \rho_1,\ldots,\rho_{M-2}.
\]
The factors \(2\) at the endpoints are positive and therefore do not
change any signs. Passing to a sublist cannot create an additional
sign change. Consequently, on every maximal block of nonzero edge
products, both \(L_+\) and \(L_-\) have at most one sign change.

Moreover, every edge product has the form
\[
    \rho_k=e^{2i\theta}c_k
    \qquad (c_k\in\mathbb R).
\]
Theorem~2.3 therefore applies separately to \(L_+\) and \(L_-\), and
gives
\[
    \spec(L_+)\subset\cX_\theta,
    \qquad
    \spec(L_-)\subset\cX_\theta.
\]
Since \(L=L_+\oplus L_-\),
\[
    \spec(L)
    =\spec(L_+)\cup\spec(L_-)
    \subset\cX_\theta.
\]
\end{proof}

\begin{proposition}
\label{prop:central-scottish-factorization}
For \(N\ge3\), set $m_N:=\left\lfloor\frac N4\right\rfloor,
 r_N:=N-4m_N=N \pmod 4\in\{0,1,2,3\}.$
There are pairwise distinct numbers
\(\eta_{1,N},\ldots,\eta_{m_N,N}>0\) such that
\begin{equation}
\label{eq:central-scottish-factorization}
 Q_{N,\pi/2}(z)
 =z^{r_N}\prod_{j=1}^{m_N}(z^4+\eta_{j,N}).
\end{equation}
In particular, every nonzero root of $Q_{N,\pi/2}$ is simple and
\[
 \{z:Q_{N,\pi/2}(z)=0\}\subset\cX_{\pi/4}.
\]
\end{proposition}

\begin{proof}
At \(\varphi=\pi/2\), the path in \Cref{lem:phase-opening} has edge
products
\begin{equation}
\label{eq:sine-cycle-weights}
 \rho_{j,N}^{\circ}
 =-\frac{i}{2}\sin\frac{2\pi j}{N},
 \qquad 1\le j<N.
\end{equation}
Suppose first that \(N=2s+1\) is odd.  For \(1\le j\le s\), put
\[
 a_{j,N}:=\frac12\sin\frac{2\pi j}{N}>0.
\]
Then
\(\rho_{j,N}^{\circ}=-ia_{j,N}\) and
\(\rho_{N-j,N}^{\circ}=ia_{j,N}\).  Thus the path satisfies
\eqref{eq:mirrored-rho} with \(\theta=-\pi/4\).  If \(N=4m+1\),
\eqref{eq:mirror-factor-even} gives
\[
 Q_{N,\pi/2}(z)
 =z\prod_{j=1}^{m}(z^4+\nu_{j,N}^2),
\]
whereas, if \(N=4m+3\), \eqref{eq:mirror-factor-odd} gives
\[
 Q_{N,\pi/2}(z)
 =z^3\prod_{j=1}^{m}(z^4+\nu_{j,N}^2).
\]
The positive \(\nu_{j,N}\) are pairwise distinct by
\Cref{thm:mirrored-path}.

Now let \(N=2M\), and define as above
\[
 a_{j,N}:=\frac12\sin\frac{2\pi j}{N}>0,
 \qquad 1\le j<M.
\]
The middle edge in \eqref{eq:sine-cycle-weights} also vanishes, so the
opened path splits into two paths of dimension \(M\).  Let \(J_M\) be
the real symmetric zero-diagonal Jacobi matrix with off-diagonal entries
\(\sqrt{a_{1,N}},\ldots,\sqrt{a_{M-1,N}}\).
The two components have product lists
\(-ia_{1,N},\ldots,-ia_{M-1,N}\) and their \(+i\)-phase reversal.  Hence
\begin{equation}
\label{eq:even-central-split}
 Q_{N,\pi/2}(z)
 =\det(zI-\e^{-i\pi/4}J_M)
  \det(zI-\e^{i\pi/4}J_M).
\end{equation}
The spectrum of \(J_M\) is simple and symmetric.  Writing its positive
eigenvalues as
\(0<\lambda_{1,N}<\cdots<\lambda_{\lfloor M/2\rfloor,N}\), and including
its single zero when \(M\) is odd, gives
\[
 Q_{N,\pi/2}(z)
 =
 \begin{cases}
 \displaystyle\prod_{j=1}^{m}(z^4+\lambda_{j,N}^4),
 &N=4m,\\[5pt]
 \displaystyle z^2\prod_{j=1}^{m}(z^4+\lambda_{j,N}^4),
 &N=4m+2.
 \end{cases}
\]
Combining the cases proves the factorization, with
\(\eta_{j,N}=\nu_{j,N}^2\) in odd dimension and
\(\eta_{j,N}=\lambda_{j,N}^4\) in even dimension.  Distinctness proves
simplicity, and \(z^4=-\eta_{j,N}<0\) gives the cross containment.
\end{proof}

\subsection{The \texorpdfstring{periodic \(\varphi=\pi/2\)}{periodic pi/2} matrix}
\label{sub:scottish-periodic}

Within this subsection, abbreviate the periodic matrix by
\[
 A_N:=A_N(\pi/2,0)=D_N(0)+i C_N.
\]

\begin{theorem}\label{thm:intro-scottish}
For every even \(N\ge4\),
\begin{equation}
\label{eq:intro-two diagonal lines}
 \spec(A_N)\subset\cX_{\pi/4}
 =\e^{\pi i/4}\R\cup\e^{-\pi i/4}\R.
\end{equation}
If \(N=4m\), then there are numbers \(\tau_{r,N}>0\) for
\(1\le r\le m-1\) such that
\begin{align*}
 A_N|_{\cG_0(A_N)}&\sim J_3(0)\oplus J_1(0) \text{ and }
 \det(zI-A_N)=z^4\prod_{r=1}^{m-1}(z^4+\tau_{r,N}),
 \qquad \tau_{r,N}>0.
\end{align*}
If \(N=4m+2\), then there are numbers \(\tau_{r,N}>0\) for
\(1\le r\le m\) such that
\begin{align*}
 A_N|_{\cG_0(A_N)}&\sim J_2(0) \text{ and }
 \det(zI-A_N)=z^2\prod_{r=1}^{m}(z^4+\tau_{r,N}),
 \qquad \tau_{r,N}>0.
\end{align*}
\end{theorem}

\begin{proof}
By \eqref{eq:scottish-products}, the products \(\rho_j\) are
\(\rho_{k,N}=\frac{i}{2}\cos((2k+1)\pi/N)\), and they are reflected with
opposite signs.  Thus the cyclic consequence of
\Cref{thm:one-sign-change}, \Cref{cor:cyclic-cross}, gives
\eqref{eq:intro-two diagonal lines}; this is also the case
\(\varphi=\pi/2\) of \Cref{thm:intro-unit-modulus}.  Let \(F\) be the unitary discrete Fourier transform on \(\mathbb C^N\),
defined by
\[
    F e_j
    :=\frac{1}{\sqrt N}\sum_{k\in\mathbb Z/N\mathbb Z}
      \omega^{-jk}e_k,
    \qquad
    \omega=e^{2\pi i/N}
\] and
\(Q=\diag((-1)^j)\).  Since
\[
 FD_N(0)F^*=C_N,\qquad FC_NF^*=D_N(0),\qquad
 QD_N(0)Q=D_N(0),\qquad QC_NQ=-C_N,
\]
we have
\[
 QA_NQ=A_N^*,\qquad (QF)A_N(QF)^*=i A_N.
\]
Hence, every nonzero eigenvalue occurs in a quartet
\begin{equation}
\label{eq:quartet}
 \lambda,\quad -\lambda,\quad i\lambda,\quad -i\lambda,
\end{equation}
with algebraic multiplicity.  The determinant formula proved in \Cref{thm:mirrored-path}
\eqref{eq:mirror-factor-even}--\eqref{eq:mirror-factor-odd} then
shows that every nonzero quartet contributes a factor \(z^4+\tau\) with
\(\tau>0\).  The remaining powers of \(z\), and the corresponding Jordan
blocks at zero, are established in \Cref{app:scottish-zero}.  Combining that
calculation with \eqref{eq:quartet} gives the two stated factorizations.
\end{proof}
The next result helps us locate the spectrum for tridiagonal matrices if they are almost hermitian. 
\begin{lemma}
\label{lem:one-defective-cycle}
Let \(C\in M_M(\mathbb C)\), \(M\ge3\), be cyclic tridiagonal with real
diagonal.  Write \(u_j=C_{j,j+1}\), \(\ell_j=C_{j+1,j}\), with indices
understood modulo \(M\), and \(p_j=u_j\ell_j\).  Suppose that
\(p_j\in\mathbb R\setminus\{0\}\), exactly one product \(p_{j_0}\) is
negative, all the others are positive, and
\[
 \prod_{j=0}^{M-1}|u_j|
 =
 \prod_{j=0}^{M-1}|\ell_j|.
\]
Then
\begin{equation}
\label{eq:defective-cycle-strip}
 \spec C\subset
 \left\{w\in\mathbb C:\ |\Im w|\le\sqrt{|p_{j_0}|}\right\}.
\end{equation}
\end{lemma}

\begin{proof}
Choose positive numbers \(s_j\), cyclically indexed, by \(s_0=1\) and
\[
 \frac{s_{j+1}}{s_j}
 =\sqrt{\frac{|\ell_j|}{|u_j|}}.
\]
The product hypothesis is exactly the consistency condition \(s_M=s_0\).
For \(S=\diag(s_0,\ldots,s_{M-1})\), the two opposite off-diagonal entries for
\(j\) of \(\widetilde C=S^{-1}CS\) have the same modulus
\(\sqrt{|p_j|}\).  Since their product is real, they are complex conjugates
when \(p_j>0\), and negatives of complex conjugates when \(p_j<0\).
Delete the unique negative edge from \(\widetilde C\), obtaining a
Hermitian matrix \(H\).  Then
\[
 \widetilde C=H+E,\qquad
 \|E\|=\sqrt{|p_{j_0}|}.
\]
Because \(H\) is Hermitian, if
\(\operatorname{dist}(w,\mathbb R)>\|E\|\), then
\(\|(w-H)^{-1}E\|<1\), so \(w-H-E\) is invertible by the Neumann series.
Thus every eigenvalue of \(\widetilde C\), and hence of \(C\), satisfies
\eqref{eq:defective-cycle-strip}.
\end{proof}

\begin{corollary}
\label{cor:periodic-phase-dichotomy}
Let \(N=4m+2\ge6\), and consider the periodic \(\varphi=\pi/2\) family
\(A_N(\pi/2,\vartheta)\).  Define
\[
 d_N(\vartheta):=\operatorname{dist}\!\left(
 \vartheta,\frac{2\pi}{N}\mathbb Z\right)\in[0,\pi/N]
\]
and
\begin{equation*}
 \gamma_N(\vartheta)
 :=
 \frac12\sqrt{\sin d_N(\vartheta)\,
 \sin\!\left(\frac{2\pi}{N}-d_N(\vartheta)\right)}.
\end{equation*}
Then every \(z\in\spec A_N(\pi/2,\vartheta)\) satisfies
\begin{equation}
\label{eq:z2-strip}
 |\Re z^2|
 \le\gamma_N(\vartheta)
 \le\frac12\sin\frac{\pi}{N}.
\end{equation}
Consequently,
\begin{equation}
\label{eq:two diagonal lines-distance-global}
 \operatorname{dist}(z,\cX_{\pi/4})
 \le\sqrt{\frac{\gamma_N(\vartheta)}2}
 \le\frac12\sqrt{\sin\frac{\pi}{N}}.
\end{equation}
If \(z\ne0\), one also has
\begin{equation}
\label{eq:two diagonal lines-distance-away}
 \operatorname{dist}(z,\cX_{\pi/4})
 \le \frac{\gamma_N(\vartheta)}{\sqrt2\,|z|}
 \le \frac{\sin(\pi/N)}{2\sqrt2\,|z|}.
\end{equation}
At the aligned phases,
\[
 \vartheta\in2\pi\mathbb Z/N
 \quad\Longrightarrow\quad
 \spec A_N(\pi/2,\vartheta)\subset\cX_{\pi/4}.
\]
At every other phase,
\[
 \vartheta\notin2\pi\mathbb Z/N
 \quad\Longrightarrow\quad
 \spec A_N(\pi/2,\vartheta)\cap\cX_{\pi/4}=\varnothing.
\]
More precisely, with the positive numbers \(\tau_{r,N}\) from
\Cref{thm:intro-scottish},
\begin{equation}
\label{eq:4m2-phase-dichotomy}
 \chi_{N,\pi/2,\vartheta}^{1}(z)
 =
 z^2\prod_{r=1}^{m}(z^4+\tau_{r,N})
 +2^{1-N}\bigl(1-\cos(N\vartheta)\bigr).
\end{equation}
\end{corollary}

\begin{proof}
The statement at the aligned phases
\[
    \vartheta\in \frac{2\pi}{N}\mathbb Z
\]
is the case \(\varphi=\pi/2\) of
\Cref{thm:intro-unit-modulus}.  Since the spectrum is invariant under
\(\vartheta\mapsto\vartheta+2\pi/N\), we may replace \(\vartheta\) by
its representative
\[
    \vartheta_0\in[0,2\pi/N].
\]
If \(\vartheta_0=0\) or \(2\pi/N\), then
\(\gamma_N(\vartheta)=0\), and the desired estimates follow from the
exact containment
\[
    \spec A_N(\pi/2,\vartheta_0)\subset \cX_{\pi/4}.
\]
We therefore assume throughout the rest of the proof that
\[
    0<\vartheta_0<\frac{2\pi}{N}.
\]

Write
\[
    N=4m+2=2M,\qquad M=2m+1.
\]
In particular, \(M\) is odd.  In the quadratic-phase basis
\eqref{eq:parity-chirp-basis}, define
\[
    c_k:=
    \frac12\cos\!\left(
        \frac{(2k+1)\pi}{N}-\vartheta_0
    \right),
    \qquad k\in\mathbb Z/N\mathbb Z.
\]
At \(\varphi=\pi/2\), equations
\eqref{eq:tridiagonal-form}--\eqref{eq:general-phase-product} give
\[
    \rho_k
    =\alpha_{k,N}\beta_{k+1,N}
    =ic_k.
\]

We first square the matrix.  Every application of
\(\widehat A_N(\pi/2,\vartheta_0)\) changes the index by \(1\), so the
matrix exchanges the even and odd index subspaces.  Since \(N=2M\),
each of these subspaces has dimension \(M\).  Ordering first the even
basis vectors and then the odd ones therefore gives
\[
    \widehat A_N(\pi/2,\vartheta_0)
    =
    \begin{pmatrix}
        0&U\\
        V&0
    \end{pmatrix},
\]
with \(U,V\in\mathbb C^{M\times M}\), and hence
\[
    \widehat A_N(\pi/2,\vartheta_0)^2
    =
    \begin{pmatrix}
        UV&0\\
        0&VU
    \end{pmatrix}.
\]
Thus the square preserves parity.

It is useful to see explicitly what the two squared blocks look like.
From
\[
    \widehat A_N\psi_k
    =
    \alpha_{k,N}\psi_{k+1}
    +\beta_{k,N}\psi_{k-1},
\]
we obtain
\[
\begin{split}
    \widehat A_N^2\psi_k
    &=
    \alpha_{k,N}\alpha_{k+1,N}\psi_{k+2}
    +\bigl(
        \alpha_{k,N}\beta_{k+1,N}
        +\beta_{k,N}\alpha_{k-1,N}
      \bigr)\psi_k
    +\beta_{k,N}\beta_{k-1,N}\psi_{k-2}.
\end{split}
\]
Since $\alpha_{k,N}\beta_{k+1,N}=ic_k,
    \beta_{k,N}\alpha_{k-1,N}=ic_{k-1},$
the diagonal coefficient of \(-i\widehat A_N^2\) at \(k\) is $c_k+c_{k-1}\in\mathbb R.$
Consequently
\[
    K_0:=-iUV,\qquad K_1:=-iVU
\]
are \(M\times M\) cyclic tridiagonal matrices with real diagonal.

We next compute their edge products.  In the even block, the two
entries connecting the sites corresponding to \(2r\) and \(2r+2\)
have product
\[
\begin{split}
    (-i)^2
    \bigl(\alpha_{2r,N}\alpha_{2r+1,N}\bigr)
    \bigl(\beta_{2r+2,N}\beta_{2r+1,N}\bigr)
    &=
    -\bigl(\alpha_{2r,N}\beta_{2r+1,N}\bigr)
     \bigl(\alpha_{2r+1,N}\beta_{2r+2,N}\bigr)\\
    &=
    -(ic_{2r})(ic_{2r+1})\\
    &=
    c_{2r}c_{2r+1}.
\end{split}
\]
Similarly, the edge products in the odd block are $c_{2r+1}c_{2r+2}.$

We now verify the balancing hypothesis in
\Cref{lem:one-defective-cycle}.  At \(\varphi=\pi/2\), since
\(\delta_N=1/2\),
\[
    \alpha_{k,N}
    =
    \frac{e^{i\vartheta_0}
          +i\omega^{k+1/2}}{2}.
\]
The numbers \(\omega^{k+1/2}\), \(0\leq k<N\), are precisely the
\(N\) roots of \(\zeta^N=-1\).  Because \(N\) is even,
\[
    \prod_{\zeta^N=-1}(x+y\zeta)=x^N+y^N.
\]
Taking \(x=e^{i\vartheta_0}\) and \(y=i\) gives
\[
    \prod_{k=0}^{N-1}\alpha_{k,N}
    =
    2^{-N}\bigl(e^{iN\vartheta_0}+i^N\bigr).
\]
The same calculation for \(\beta_{k,N}\) gives
\begin{equation}
\label{eq:alpha-beta-products}
    \prod_{k=0}^{N-1}\alpha_{k,N}
    =
    2^{-N}\bigl(e^{iN\vartheta_0}+i^N\bigr),
    \qquad
    \prod_{k=0}^{N-1}\beta_{k,N}
    =
    2^{-N}\bigl(e^{-iN\vartheta_0}+i^N\bigr).
\end{equation}
Since \(N\) is even, \(i^N\in\mathbb R\), and the two quantities on
the right are complex conjugates.  In particular,
\[
    \prod_{k=0}^{N-1}|\alpha_{k,N}|
    =
    \prod_{k=0}^{N-1}|\beta_{k,N}|.
\]

For either squared block, the product of the entries in one direction
around the cycle is, up to factors of modulus one, the product of all
the \(\alpha_{k,N}\), while the product in the opposite direction is,
again up to factors of modulus one, the product of all the
\(\beta_{k,N}\).  Indeed, in the even block the forward coefficients
are
\[
    -i\alpha_{2r,N}\alpha_{2r+1,N},
    \qquad 0\leq r<M,
\]
whose product contains every \(\alpha_{k,N}\) exactly once; the
backward coefficients similarly contain every \(\beta_{k,N}\)
exactly once.  The odd block has the same property, with the indices
shifted by one.  Hence both \(K_0\) and \(K_1\) satisfy the balancing
condition in \Cref{lem:one-defective-cycle}.

It remains to understand the signs of their edge products.  The
angles
\[
    \frac{(2k+1)\pi}{N}-\vartheta_0
\]
advance by \(2\pi/N\) as \(k\) increases.  Since
\(0<\vartheta_0<2\pi/N\), the sequence \(c_k\) changes sign exactly
twice around the full cycle, once when the sampled angle crosses
\(\pi/2\) and once when it crosses \(3\pi/2\).

The first sign change occurs between \(k=m\) and \(k=m+1\).  Indeed,
using \(N=4m+2\),
\[
    \frac{(2m+1)\pi}{N}=\frac{\pi}{2},
\]
and therefore
\[
    c_m
    =
    \frac12\cos\!\left(\frac{\pi}{2}-\vartheta_0\right)
    =
    \frac12\sin\vartheta_0>0,
\]
whereas
\[
\begin{split}
    c_{m+1}
    &=
    \frac12\cos\!\left(
        \frac{\pi}{2}+\frac{2\pi}{N}-\vartheta_0
    \right)=
    -\frac12\sin\!\left(
        \frac{2\pi}{N}-\vartheta_0
    \right)<0.
\end{split}
\]
The second sign change occurs between \(k=3m+1\) and \(k=3m+2\).
The two sign-changing pairs are therefore separated by
\[
    (3m+1)-m=2m+1=M
\]
indices.  Since \(M\) is odd, one of these pairs starts at an even
index and the other starts at an odd index.  Consequently exactly one
negative adjacent product occurs among $c_{2r}c_{2r+1},$
and exactly one occurs among $c_{2r+1}c_{2r+2}.$
In other words, each of \(K_0\) and \(K_1\) has exactly one negative
edge product, while all of its other edge products are positive.

The two negative products have the same absolute value.  From the
formulas above,
\[
    |c_mc_{m+1}|
    =
    \frac14
    \sin\vartheta_0
    \sin\!\left(\frac{2\pi}{N}-\vartheta_0\right).
\]
Since
\[
    d_N(\vartheta)
    =
    \min\!\left\{
        \vartheta_0,\frac{2\pi}{N}-\vartheta_0
    \right\},
\]
and the product of the two sine factors is symmetric under
\(\vartheta_0\mapsto2\pi/N-\vartheta_0\), this equals $\gamma_N(\vartheta)^2.$
We may therefore apply \Cref{lem:one-defective-cycle} to both squared
blocks.  It gives
\[
    \spec K_\nu
    \subset
    \{w\in\mathbb C:|\Im w|\leq\gamma_N(\vartheta)\},
    \qquad \nu=0,1.
\]

Now let $z\in\spec A_N(\pi/2,\vartheta_0).$
The change to the quadratic-phase basis is unitary, so \(z\) is also
an eigenvalue of \(\widehat A_N(\pi/2,\vartheta_0)\).  Therefore
\(z^2\) is an eigenvalue of its square, and hence belongs to the
spectrum of one of \(UV\) and \(VU\).  It follows that
\[
    -iz^2\in\spec K_0\cup\spec K_1.
\]
The preceding strip estimate gives
\[
    |\Re z^2|
    =
    |\Im(-iz^2)|
    \leq\gamma_N(\vartheta),
\]
which is the first inequality in \eqref{eq:z2-strip}.

For the second inequality, put \(h=2\pi/N\).  For
\(0\leq t\leq h\),
\[
    \sin t\,\sin(h-t)
    =
    \frac{\cos(2t-h)-\cos h}{2},
\]
which is maximal at \(t=h/2=\pi/N\).  Hence
\[
    \gamma_N(\vartheta)
    \leq
    \frac12\sin\frac{\pi}{N}.
\]
This completes the proof of \eqref{eq:z2-strip}.

We next translate the bound on \(z^2\) into a distance from the two
diagonal lines.  Write \(z=x+iy\), and denote the distances from
\(z\) to the two lines in \(\cX_{\pi/4}\) by
\[
    d_1=\frac{|x-y|}{\sqrt2},
    \qquad
    d_2=\frac{|x+y|}{\sqrt2}.
\]
Since
\[
    \Re z^2=x^2-y^2=(x-y)(x+y),
\]
we have
\[
    d_1d_2=\frac{|\Re z^2|}{2},
\]
and directly
\[
    d_1^2+d_2^2=x^2+y^2=|z|^2.
\]
Let
\[
    d:=\min(d_1,d_2)
    =\operatorname{dist}(z,\cX_{\pi/4}),
    \qquad
    D:=\max(d_1,d_2).
\]
Then \(d\leq D\), so
\[
    d^2\leq dD=d_1d_2
    \leq\frac{\gamma_N(\vartheta)}{2}.
\]
Therefore
\[
    \operatorname{dist}(z,\cX_{\pi/4})
    \leq
    \sqrt{\frac{\gamma_N(\vartheta)}{2}},
\]
and the second bound in
\eqref{eq:two diagonal lines-distance-global} follows from
\(\gamma_N(\vartheta)\leq\frac12\sin(\pi/N)\).

If \(z\neq0\), then
\[
    D^2\geq\frac{d_1^2+d_2^2}{2}
    =\frac{|z|^2}{2},
\]
so \(D\geq|z|/\sqrt2\).  Using again
\(dD=d_1d_2\leq\gamma_N(\vartheta)/2\), we obtain
\[
    \operatorname{dist}(z,\cX_{\pi/4})
    =d
    \leq
    \frac{\gamma_N(\vartheta)}{\sqrt2\,|z|},
\]
which proves \eqref{eq:two diagonal lines-distance-away}.

It remains to prove that, away from the aligned phases, no eigenvalue
lies exactly on the two diagonal lines.  Combining
\eqref{eq:Chambers-phase} with the \(N=4m+2\) factorization in
\Cref{thm:intro-scottish} gives
\eqref{eq:4m2-phase-dichotomy}, namely
\[
    \chi^1_{N,\pi/2,\vartheta}(z)
    =
    z^2\prod_{r=1}^m(z^4+\tau_{r,N})
    +
    2^{1-N}\bigl(1-\cos(N\vartheta)\bigr).
\]
Suppose that \(z\in\cX_{\pi/4}\).  Then
\[
    z^2\in i\mathbb R,
    \qquad
    z^4\in\mathbb R.
\]
Since every \(\tau_{r,N}\) is real and positive,
\[
    z^2\prod_{r=1}^m(z^4+\tau_{r,N})
    \in i\mathbb R.
\]
On the other hand, if
\(\vartheta\notin2\pi\mathbb Z/N\), then
\[
    1-\cos(N\vartheta)>0,
\]
so the second term in
\eqref{eq:4m2-phase-dichotomy} is strictly positive and real.
A purely imaginary number and a strictly positive real number cannot
sum to zero.  Hence
\[
    \chi^1_{N,\pi/2,\vartheta}(z)\neq0
    \qquad\text{for every }z\in\cX_{\pi/4}.
\]
Thus
\[
    \spec A_N(\pi/2,\vartheta)\cap\cX_{\pi/4}=\varnothing
\]
whenever
\(\vartheta\notin2\pi\mathbb Z/N\), completing the proof.
\end{proof}

\begin{theorem}
\label{thm:odd-scottish}
Let \(N\ge3\) be odd, and put $m:=\left\lfloor\frac N4\right\rfloor, 
 r_N:=N-4m\in\{1,3\}.$
Then
\begin{equation}
\label{eq:odd-scottish-central}
 \chi_{N,\pi/2,\pi/2}^{i^N}(z)=Q_{N,\pi/2}(z),
 \qquad
 \det(zI-B_N)=(-2)^NQ_{N,\pi/2}(-z/2).
\end{equation}
\begin{equation}
\label{eq:odd-B-factorization}
 \det(zI-B_N)
 =z^{r_N}\prod_{j=1}^{m}(z^4+16\eta_{j,N}).
\end{equation}
Consequently,
\begin{equation}
\label{eq:odd-B-cross}
 \spec(B_N)\subset\cX_{\pi/4},
\end{equation}
every nonzero eigenvalue is simple, and
\begin{equation}
\label{eq:odd-B-zero-Jordan}
 B_N|_{\cG_0(B_N)}
 \sim
 \begin{cases}
 J_1(0),&N\equiv1\pmod4,\\
 J_3(0),&N\equiv3\pmod4.
 \end{cases}
\end{equation}
Moreover,
\begin{equation}
\label{eq:quarter-turn-spectrum}
 \spec(B_N)=i\,\spec(B_N)
\end{equation}
with algebraic multiplicity.
\end{theorem}

\begin{proof}
For odd \(N\), the two cosine terms in
\eqref{eq:twisted-Chambers-Q} vanish at
\((\varphi,\vartheta,\tau)=(\pi/2,\pi/2,i^N)\), proving the first identity in
\eqref{eq:odd-scottish-central}.  The conjugation
\(-G_N^{-1}B_NG_N=2A_N^{i^N}(\pi/2,\pi/2)\) gives the second.
Substituting \eqref{eq:central-scottish-factorization} proves
\eqref{eq:odd-B-factorization}, hence the cross containment, simplicity,
and quarter-turn symmetry.

It remains to determine the Jordan block at zero.  Put
\(s_j=\sin(2\pi j/N)\).  The equation \(B_Nx=0\) is equivalent to
\begin{equation}
\label{eq:odd-B-transfer}
 \binom{x_{j+1}}{x_j}
 =
 T_j\binom{x_j}{x_{j-1}},
 \qquad
 T_j=
 \begin{pmatrix}2s_j&1\\1&0\end{pmatrix}.
\end{equation}
Let \(\mathcal M_N=T_{N-1}\cdots T_1T_0\).  Periodic solutions are in
one-to-one correspondence with fixed vectors of \(\mathcal M_N\).
Since \(\det\mathcal M_N=(-1)^N=-1\), its fixed space cannot be
two-dimensional: otherwise \(\mathcal M_N=I\), whose determinant is
\(1\).  The factorization shows that zero is an eigenvalue, so this fixed
space is nonzero.  Hence \(\dim\ker B_N=1\).  Its algebraic multiplicity
is \(r_N\), and the unique zero Jordan block has size \(r_N\).
\end{proof}

\subsection{The original Scottish flag matrix}
\label{sub:scottish-original}

\begin{theorem}
\label{thm:original-scottish}
For every \(N\ge3\),
\begin{equation}
\label{eq:original-B-cross}
 \spec(B_N)\subset\cX_{\pi/4}.
\end{equation}
Moreover,
\begin{equation}
\label{eq:B-characteristic-scaling}
 \det(zI-B_N)
 =
 \begin{cases}
 \displaystyle
 2^N\det\!\left(-\frac z2 I-A_N\right),
 &N\ \mathrm{even},\\[7pt]
 \displaystyle
 (-2)^NQ_{N,\pi/2}(-z/2),
 &N\ \mathrm{odd}.
 \end{cases}
\end{equation}
More precisely, the following four cases hold.

If \(N=4m\), then there are \(\sigma_{1,N},\ldots,\sigma_{m-1,N}>0\)
such that
\[
 B_N|_{\cG_0(B_N)}\sim J_3(0)\oplus J_1(0),
 \qquad
 \det(zI-B_N)=z^4\prod_{r=1}^{m-1}(z^4+\sigma_{r,N}).
\]
If \(N=4m+2\), then there are numbers \(\sigma_{r,N}>0\) for
\(1\le r\le m\) such that
\[
 B_N|_{\cG_0(B_N)}\sim J_1(0)\oplus J_1(0),
 \qquad
 \det(zI-B_N)=z^2\prod_{r=1}^{m}(z^4+\sigma_{r,N}).
\]
If \(N=4m+1\), then there are pairwise distinct
\(\sigma_{1,N},\ldots,\sigma_{m,N}>0\) such that
\[
 B_N|_{\cG_0(B_N)}\sim J_1(0),
 \qquad
 \det(zI-B_N)=z\prod_{r=1}^{m}(z^4+\sigma_{r,N}).
\]
If \(N=4m+3\), then there are pairwise distinct
\(\sigma_{1,N},\ldots,\sigma_{m,N}>0\) such that
\[
 B_N|_{\cG_0(B_N)}\sim J_3(0),
 \qquad
 \det(zI-B_N)=z^3\prod_{r=1}^{m}(z^4+\sigma_{r,N}).
\]
In odd dimension every nonzero eigenvalue is simple.  Thus the Scottish
flag phenomenon holds for every \(N\ge3\).
\end{theorem}

\begin{proof}
Suppose first that \(N\) is even.  The conjugation of
\Cref{sec:introduction} gives
\(-G_N^{-1}B_NG_N=2A_N^{i^N}(\pi/2,\pi/2)\).  By
\eqref{eq:twist-cancellation-intro}, this twisted matrix and the periodic
\(\varphi=\pi/2\) matrix \(A_N\) have the same characteristic polynomial.  Since
\(N\) is even, scaling the spectral parameter gives
\eqref{eq:B-characteristic-scaling}.  The two diagonal lines and the stated
factorizations now follow from \Cref{thm:intro-scottish}.

If \(4\mid N\), the conjugated matrix is itself a translated periodic matrix,
so its zero Jordan form is the one in \Cref{thm:intro-scottish}.  If
\(N\equiv2\pmod4\), the zero eigenvalue has algebraic multiplicity two by the
characteristic identity and geometric multiplicity two by
\Cref{prop:antiperiodic-zero}.  Its two Jordan blocks at zero therefore both have size one, giving
\(J_1(0)\oplus J_1(0)\).

For odd \(N\), all assertions follow from \Cref{thm:odd-scottish}, with
\(\sigma_{j,N}=16\eta_{j,N}\).
\end{proof}

\section{Limiting eigenvalue distribution}\label{sec:distribution}

We need the following limit for real symmetric tridiagonal matrices, together with its quantitative form.  Let
\begin{equation*}
 J_n=\begin{pmatrix}
 b_{1,n}&a_{1,n}\\
 a_{1,n}&b_{2,n}&a_{2,n}\\
 &\ddots&\ddots&\ddots\\
 &&a_{n-1,n}&b_{n,n}
 \end{pmatrix},
 \qquad a_{j,n}\ge0.
\end{equation*}
For \(r\geq1\) and \(1\leq j\leq n\), let
\(\mathcal W_r(j)\) be the set of sequences
\[
(j_0,\ldots,j_r),\qquad
j_0=j_r=j,\qquad |j_{\ell+1}-j_\ell|\leq1.
\]
The weight of such a walk is
\[
w(j_0,\ldots,j_r)
:=\prod_{\ell=0}^{r-1}(J_n)_{j_\ell,j_{\ell+1}}.
\]
Then
\[
(J_n^r)_{jj}
=\sum_{\gamma\in\mathcal W_r(j)}w(\gamma).
\]

\begin{proposition}\label{prop:Jacobi-symbol}
Suppose \(a,b\in C([0,1];\R)\) and
\begin{equation*}
 \max_j\left|a_{j,n}-a(j/n)\right|
 +\max_j\left|b_{j,n}-b(j/n)\right|\longrightarrow0.
\end{equation*}
Then, for every continuous \(f\),
\begin{equation}
\label{eq:Jacobi-limit}
 \frac1n\tr f(J_n)\longrightarrow
 \frac1\pi\int_0^1\int_0^\pi
 f\bigl(b(x)+2a(x)\cos\vartheta\bigr)\,d\vartheta\,dx.
\end{equation}
\end{proposition}

\begin{proof}
We first prove the statement for monomials
\[
    f(t)=t^p, \qquad p\in\mathbb Z_{\geq 0}.
\]
The case \(p=0\) is immediate, so assume \(p\geq1\).

For \(1\leq j\leq n\), matrix multiplication gives
\[
    (J_n^p)_{jj}
    =
    \sum_{j_1,\ldots,j_{p-1}}
    (J_n)_{j,j_1}(J_n)_{j_1,j_2}\cdots
    (J_n)_{j_{p-1},j}.
\]
Since \(J_n\) is tridiagonal, a term in this sum can be nonzero only
when two consecutive indices differ by at most one. Thus
\((J_n^p)_{jj}\) is the sum of the weights of all closed walks
\[
    j=j_0,j_1,\ldots,j_p=j,
    \qquad |j_{\ell+1}-j_\ell|\leq1,
\]
of length \(p\) based at \(j\), as we discussed just before the statement of the proposition. A step \(k\to k+1\) or
\(k+1\to k\) contributes the corresponding off-diagonal coefficient
\(a_{k,n}\), while a step \(k\to k\) contributes \(b_{k,n}\).

Fix \(p\). A walk of length \(p\) based at \(j\) only visits indices
within distance \(p\) of \(j\). Put
\[
    x_j:=\frac{j}{n}.
\]
For every fixed \(|\ell|\leq p\), the assumptions and the uniform
continuity of \(a\) and \(b\) give
\[
    a_{j+\ell,n}=a(x_j)+o(1),
    \qquad
    b_{j+\ell,n}=b(x_j)+o(1),
\]
uniformly in \(j\), as long as the displayed indices lie in
\(\{1,\ldots,n\}\). Indeed,
\[
    \left|a_{j+\ell,n}-a(x_j)\right|
    \leq
    \left|a_{j+\ell,n}
          -a\!\left(\frac{j+\ell}{n}\right)\right|
    +
    \left|a\!\left(\frac{j+\ell}{n}\right)-a(x_j)\right|,
\]
and both terms tend to zero uniformly in \(j\); the same argument
applies to \(b\).

Suppose now that \(p<j<n-p\), so that no walk of length \(p\) based
at \(j\) reaches the boundary. Replacing all coefficients encountered
by such a walk by \(a(x_j)\) and \(b(x_j)\) changes its weight by
\(o(1)\), uniformly in \(j\). There are at most \(3^p\) possible
step sequences, so
\begin{equation}
\label{eq:local-walk-approximation}
    (J_n^p)_{jj}
    =
    [z^0]\bigl(
        b(x_j)+a(x_j)(z+z^{-1})
    \bigr)^p
    +o(1),
\end{equation}
uniformly for \(p<j<n-p\).

To see the appearance of the constant Laurent coefficient, associate
the factors $a(x_j)z, b(x_j), a(x_j)z^{-1}$
with a step to the right, a stationary step, and a step to the left,
respectively. A term corresponding to a walk with \(r_+\) right steps
and \(r_-\) left steps contains the factor \(z^{r_+-r_-}\). The walk
is closed exactly when \(r_+=r_-\), so the sum of the weights of the
closed walks is precisely the coefficient of \(z^0\).

For any Laurent polynomial \(F\),
\[
    [z^0]F(z)
    =
    \frac{1}{2\pi}\int_0^{2\pi}F(e^{i\vartheta})\,d\vartheta.
\]
Since
\[
    e^{i\vartheta}+e^{-i\vartheta}
    =2\cos\vartheta,
\]
we obtain
\begin{align*}
    [z^0]\bigl(
        b(x)+a(x)(z+z^{-1})
    \bigr)^p
    &=
    \frac{1}{2\pi}
    \int_0^{2\pi}
    \bigl(b(x)+2a(x)\cos\vartheta\bigr)^p
    \,d\vartheta
    =
    \frac{1}{\pi}
    \int_0^\pi
    \bigl(b(x)+2a(x)\cos\vartheta\bigr)^p
    \,d\vartheta.
\end{align*}
Define
\[
    F_p(x)
    :=
    \frac{1}{\pi}
    \int_0^\pi
    \bigl(b(x)+2a(x)\cos\vartheta\bigr)^p
    \,d\vartheta.
\]
Then \eqref{eq:local-walk-approximation} says
\[
    (J_n^p)_{jj}=F_p(j/n)+o(1)
\]
uniformly for \(p<j<n-p\).

It remains to deal with the boundary indices. The assumptions imply
that the matrices \(J_n\) are uniformly bounded in operator norm.
Hence, for fixed \(p\),
\[
    |(J_n^p)_{jj}|
    \leq \|J_n\|^p
    \leq C_p
\]
uniformly in \(n\) and \(j\). There are only \(\mathcal O(p)\) indices within
distance \(p\) of the two endpoints, so their total contribution to
\(\operatorname{tr}J_n^p\) is \(\mathcal O(1)=o(n)\). Therefore
\[
    \frac{1}{n}\operatorname{tr}J_n^p
    =
    \frac{1}{n}\sum_{j=1}^n F_p(j/n)+o(1).
\]
Since \(F_p\) is continuous on \([0,1]\), the sum on the right is a
Riemann sum, and hence
\[
    \frac{1}{n}\operatorname{tr}J_n^p
    \longrightarrow
    \int_0^1 F_p(x)\,dx
    =
    \frac{1}{\pi}
    \int_0^1\int_0^\pi
    \bigl(b(x)+2a(x)\cos\vartheta\bigr)^p
    \,d\vartheta\,dx.
\]
By linearity, \eqref{eq:Jacobi-limit} follows for every polynomial.

Finally, the assumptions imply that $\sup_n\|J_n\|<\infty.$
Thus there is \(R>0\) such that $\operatorname{spec}(J_n)\subset[-R,R]$
for every \(n\), and also $b(x)+2a(x)\cos\vartheta\in[-R,R]$
after increasing \(R\) if necessary. Let \(f\in C([-R,R])\).
By the Weierstrass approximation theorem, for every
\(\varepsilon>0\) there is a polynomial \(P\) such that
\[
    \sup_{|t|\leq R}|f(t)-P(t)|<\varepsilon.
\]
It follows that
\[
    \left|
    \frac1n\operatorname{tr}f(J_n)
    -
    \frac1n\operatorname{tr}P(J_n)
    \right|
    \leq \varepsilon
\text{ while }
    \left|
    \frac1\pi\int_0^1\int_0^\pi
    \bigl(f-P\bigr)
    \bigl(b(x)+2a(x)\cos\vartheta\bigr)
    \,d\vartheta\,dx
    \right|
    \leq\varepsilon.
\]
Since the desired limit is already known for \(P\), letting
\(n\to\infty\) and then \(\varepsilon\to0\) proves
\eqref{eq:Jacobi-limit} for every continuous \(f\).
\end{proof}
A more quantitative version of this result is:
\begin{lemma}\label{lem:quantitative-Jacobi}
Fix \(p\in\Z_{\ge0}\).
\begin{enumerate}
\item If \(a,b\in C^1([0,1])\) and $\max_j|a_{j,n}-a(j/n)|+
 \max_j|b_{j,n}-b(j/n)|=\mathcal  O(n^{-1}),$
then
\begin{equation}
\label{eq:quantitative-general}
 \frac{\tr J_n^p}{n}=
 \int_0^1\int_0^\pi
 \bigl(b(x)+2a(x)\cos\vartheta\bigr)^p\frac{d\vartheta\,dx}{\pi}
 +\mathcal  O(n^{-1}).
\end{equation}
\item Suppose \(b_{j,n}=0\), \(q\in C^1([0,1])\), \(q\ge0\), and
\begin{equation}
\label{eq:edge-square-profile}
 \max_j|a_{j,n}^2-q(j/n)|=\mathcal  O(n^{-1}).
\end{equation}
Then \(\tr J_n^{2p+1}=0\) and
\begin{equation}
\label{eq:quantitative-edge-square}
 \frac{\tr J_n^{2p}}{n}
 =\int_0^1\int_0^\pi \bigl(2\sqrt{q(x)}\cos\vartheta\bigr)^{2p}\,\frac{d\vartheta\,dx}{\pi}
 +\mathcal  O(n^{-1}).
\end{equation}
\end{enumerate}
The asymptotic formulas \eqref{eq:quantitative-general} and
\eqref{eq:quantitative-edge-square} remain valid under uniformly bounded
perturbations of rank \(\mathcal  O(1)\).  For the unperturbed zero-diagonal
matrix all odd traces vanish exactly; after such a perturbation \(E_n\),
\begin{equation}
\label{eq:odd-perturbed}
 \frac{\tr(J_n+E_n)^{2p+1}}{n}=\mathcal  O(n^{-1}).
\end{equation}
\end{lemma}

\begin{proof}
Let \(r\) denote the power under consideration.  For an interior walk
\(\gamma\in\mathcal W_r(j)\), every visited index has the form \(j+\ell\)
with \(|\ell|\le r\).  The assumptions and the mean-value theorem give
\begin{align*}
 a_{j+\ell,n}=a(j/n)+\mathcal  O(n^{-1}), \ b_{j+\ell,n}=b(j/n)+\mathcal  O(n^{-1}),
 \qquad |\ell|\le r,
\end{align*}
uniformly in \(j\).  If \(x_1,\ldots,x_r\) and
\(y_1,\ldots,y_r\) are uniformly bounded, then
\begin{equation*}
 \prod_{k=1}^r x_k-\prod_{k=1}^r y_k
 =\sum_{k=1}^r
   \left(\prod_{\ell<k}x_\ell\right)(x_k-y_k)
   \left(\prod_{\ell>k}y_\ell\right).
\end{equation*}
Hence replacing nearby coefficients by their value at \(j/n\) in one walk changes its weight by
\(\mathcal  O(n^{-1})\).  Since the number of possible step patterns is at most \(3^r\),
\begin{equation*}
 (J_n^r)_{jj}
 =\frac1\pi\int_0^\pi
   \bigl(b(j/n)+2a(j/n)\cos\vartheta\bigr)^r
   \,d\vartheta+\mathcal  O(n^{-1})
\end{equation*}
uniformly for \(r<j<n-r\).  The boundary contains at most \(2r\) indices,
so its normalized contribution is \(\mathcal  O(n^{-1})\).  Define
\[
 F(x):=\frac1\pi\int_0^\pi
 \bigl(b(x)+2a(x)\cos\vartheta\bigr)^r\,d\vartheta.
\]
Then \(F\in C^1([0,1])\), and the first-order Riemann-sum estimate gives
\begin{equation*}
 \frac1n\sum_{j=1}^n F(j/n)
 =\int_0^1F(x)\,dx+\mathcal  O(n^{-1}).
\end{equation*}
Combining the last three displays proves \eqref{eq:quantitative-general}.

Suppose now that the diagonal is zero.  With $\Xi=\diag(1,-1,1,-1,\ldots),$ we have \(\Xi J_n\Xi=-J_n\), and therefore $\tr J_n^{2p+1}
 =\tr(\Xi J_n^{2p+1}\Xi)
 =-\tr J_n^{2p+1}=0.$
For an even closed walk, let \(N_e^+(\gamma)\) and
\(N_e^-(\gamma)\) be the numbers of crossings of edge \(e\) in the two
directions.  Closedness implies
\begin{equation*}
 N_e^+(\gamma)=N_e^-(\gamma)
 \qquad\text{for every edge }e.
\end{equation*}
Thus its weight is
\begin{equation*}
 \prod_e a_{e,n}^{N_e^+(\gamma)+N_e^-(\gamma)}
 =\prod_e\bigl(a_{e,n}^2\bigr)^{N_e^+(\gamma)}.
\end{equation*}
The hypothesis \eqref{eq:edge-square-profile} therefore gives the same
\(\mathcal  O(n^{-1})\) estimate without requiring \(\sqrt q\) to be
\(C^1\) at its zeros.  Finally,
\begin{align*}
 \int_0^\pi(2\sqrt q\cos\vartheta)^{2p}\,\frac{d\vartheta}{\pi}
 &=4^pq^p\int_0^\pi\cos^{2p}\vartheta\,\frac{d\vartheta}{\pi}=\binom{2p}{p}q^p,
\end{align*}
which proves \eqref{eq:quantitative-edge-square}.

Let now \(E_n\) have uniformly bounded norm and rank.  For every fixed
integer \(r\ge1\),
\begin{equation*}
 (J_n+E_n)^r-J_n^r
 =\sum_{k=0}^{r-1}(J_n+E_n)^kE_nJ_n^{r-1-k}.
\end{equation*}
Each summand has rank at most \(\rank E_n=\mathcal  O(1)\) and uniformly bounded norm.
Using \(|\tr X|\le\rank(X)\|X\|\),
\begin{equation*}
 \left|\tr(J_n+E_n)^r-\tr J_n^r\right|=\mathcal  O(1).
\end{equation*}
Division by \(n\) proves stability of the normalized asymptotics.  Taking
\(r=2p+1\) and using the exact vanishing above gives
\eqref{eq:odd-perturbed}.
\end{proof}

\begin{lemma}
\label{lem:rank-BV}
Let \(H,K\in\mathbb R^{n\times n}\) be symmetric, suppose
\(\operatorname{rank}(H-K)\le r\), and assume
\(\spec(H)\cup\spec(K)\subset[-R,R]\).  If
\(f\colon[-R,R]\to\mathbb C\) is absolutely continuous, then
\[
 \left|\tr f(H)-\tr f(K)\right|
 \le r\int_{-R}^{R}|f'(t)|\,dt.
\]
In particular, for each fixed \(m\ge1\), the functions
\[
 f_m^+(x)=(x_+)^m,
 \qquad
 f_m^-(x)=((-x)_+)^m,
 \qquad x_+:=\max\{x,0\},
\]
are absolutely continuous with
\(\int_{-R}^{R}|(f_m^\pm)'(t)|\,dt\le R^m\).  Hence replacing a
Hermitian matrix by a uniformly bounded rank-\(\mathcal  O(1)\) perturbation
changes the corresponding normalized traces by \(\mathcal  O(n^{-1})\),
provided the two spectra remain in a common bounded interval.
\end{lemma}

\begin{proof}
Let $N_\bullet(t)=\#\{j:\lambda_j(\bullet)\le t\}$ with
$\bullet\in\{K,H\}$.  The rank inequality for Hermitian matrices gives
\[
 |N_H(t)-N_K(t)|\le r,
 \qquad t\in\mathbb R.
\]
Since \(H\) and \(K\) have the same dimension, integration by parts gives
\[
 \tr f(H)-\tr f(K)
 =-\int_{-R}^{R}f'(t)\bigl(N_H(t)-N_K(t)\bigr)\,dt.
\]
Taking absolute values proves the stated estimate.  The final assertion
follows from the displayed derivative bound for \(f_m^\pm\).
\end{proof}
\begin{lemma}
\label{lem:limit-Jacobi-reduction}
Fix \(\varphi\in[-\pi,\pi]\), and let \(N\to\infty\) through even
integers.  After the reflection and odd--even decompositions, the
spectrum of \(A_N(\varphi,0)\), apart from \(\mathcal O(1)\) zero
eigenvalues, is obtained from real symmetric matrices
\(H_{+,N}\) and \(H_{-,N}\) through
\[
    z^2=e^{i\varphi}\mu,
    \qquad
    \mu\in\spec(H_{+,N})\cup\spec(H_{-,N}).
\]
Moreover,
\[
    H_{+,N}\oplus H_{-,N}
    =
    K_N^+\oplus(-K_N^-)+R_N,
    \qquad
    \rank R_N=\mathcal O(1),
\]
where \(K_N^\pm\) are direct sums of a uniformly bounded number of
positive semidefinite Jacobi matrices with limiting coefficient
profiles
\[
    b_\pm(u)=2q_\pm(u),
    \qquad
    a_\pm(u)=q_\pm(u),
\]
with
\[
    q_+(u)=c_\varphi(u),\quad 0\le u\le u_0,
    \qquad
    q_-(u)=-c_\varphi(u),\quad u_0\le u\le\pi.
\]
\end{lemma}

\begin{proof}
Set
\[
 u_{k,N}:=\frac{(2k+1)\pi}{N},
 \qquad
 c_{k,N}:=c_\varphi(u_{k,N}),
 \qquad
 d_{k,N}:=|c_{k,N}|,
 \qquad 0\leq k<M.
\]
Since \(c_\varphi\) is strictly decreasing on \([0,\pi]\) and vanishes
at \(u_0\), at most one of the \(c_{k,N}\) is zero.

Suppose first that an edge product in one of the reflection blocks
vanishes.  Replace the two entries on that edge by zero.  By the
continuant recurrence \eqref{eq:continuant}, this leaves the
characteristic polynomial unchanged and merely cuts the path into
two components.  We make this replacement whenever necessary, so
that every component subsequently symmetrized has nonzero edge
products.

Apply
\Cref{lem:diagonal-similarity,prop:odd-even}
to each component.  If there is no zero product, choose, as in the
proof of \Cref{thm:one-sign-change}, the one of \(B^TB\) and \(BB^T\)
whose off-diagonal entries avoid the unique sign change.  If a zero
edge was cut, the resulting components have constant sign, and we
choose the two possible parity pairings in the two reflection
sectors.

For \(\sigma\in\{+,-\}\), denote the resulting matrix, or direct sum
of matrices, by \(G_{\sigma,N}\).  All its entries have the common
factor \(e^{i\varphi}\), so write
\[
 G_{\sigma,N}
 =
 e^{i\varphi}H_{\sigma,N},
 \qquad
 H_{\sigma,N}=H_{\sigma,N}^T
 \in\mathbb R^{d_{\sigma,N}\times d_{\sigma,N}}.
\]
The odd--even decomposition changes the dimension only by a bounded
amount, hence
\[
 d_{\sigma,N}
 =
 \frac{n_{\sigma,N}}2+\mathcal O(1).
\]
Apart from the \(\mathcal O(1)\) zero eigenvalues caused by
rectangular odd--even blocks, \Cref{prop:odd-even} gives
\[
 z^2=e^{i\varphi}\mu,
 \qquad
 \mu\in\spec(H_{\sigma,N}).
\]
Thus positive \(\mu\) produces
\[
 z=\pm e^{i\varphi/2}\sqrt{\mu},
\]
whereas negative \(\mu\) produces
\[
 z=
 \pm e^{i(\varphi/2+\pi/2)}\sqrt{-\mu}.
\]

We next describe the entries of \(H_{\sigma,N}\).  Consider a maximal
run of nonzero edge factors having one sign,
\[
 I_N=\{j,j+1,\ldots,j+\ell-1\},
\]
and put
\[
 e_{r,N}:=d_{j+r-1,N},
 \qquad 1\leq r\leq\ell,
 \qquad
 e_{0,N}=e_{\ell+1,N}:=0.
\]
By \eqref{eq:gram-entries}, after removing the common factor
\(e^{i\varphi}\), the two possible odd--even matrices have entries
\begin{align}
 (K_N^{(0)})_{r,r}
 &=
 e_{2r-1,N}+e_{2r,N},
 &
 (K_N^{(0)})_{r,r+1}
 &=
 \sqrt{e_{2r,N}e_{2r+1,N}},
 \label{eq:limit-parity-zero}\\
 (K_N^{(1)})_{r,r}
 &=
 e_{2r-2,N}+e_{2r-1,N},
 &
 (K_N^{(1)})_{r,r+1}
 &=
 \sqrt{e_{2r-1,N}e_{2r,N}}.
 \label{eq:limit-parity-one}
\end{align}
The convention \(e_{0,N}=e_{\ell+1,N}=0\) is the endpoint convention
\(w_0=w_n=0\) in \Cref{prop:odd-even}.

The doubled products at the endpoints of the \(+1\)-reflection block
in \eqref{eq:fold-plus-products} affect only the endpoint rows.  If
the unique sign change lies between two nonzero products, there is
also exactly one row whose diagonal entry combines factors of
opposite sign.  Zeroing that row and its adjacent entries, while
retaining the resulting one-dimensional zero block, changes the
matrix by uniformly bounded rank.

Let \(K_{\sigma,N}^+\) and \(K_{\sigma,N}^-\) be the positive and
negative tridiagonal blocks obtained in this way, and set
\[
 H_{\sigma,N}^{(0)}
 :=
 K_{\sigma,N}^+\oplus(-K_{\sigma,N}^-).
\]
Then
\[
 K_{\sigma,N}^\pm\geq0,
 \qquad
 \rank\bigl(
 H_{\sigma,N}-H_{\sigma,N}^{(0)}
 \bigr)
 =\mathcal O(1).
\]

It remains to identify the coefficient profiles.  On a positive
component use \(q_+(u)=c_\varphi(u)\), and on a negative component
use \(q_-(u)=-c_\varphi(u)\).  If \(v_{r,N}\) is the midpoint of the
two adjacent sampling points appearing in the \(r\)-th diagonal
entry, then
\[
 u_{k+1,N}-u_{k,N}=\frac{2\pi}{N}.
\]
Since \(q_\pm\in C^1\), equations
\eqref{eq:limit-parity-zero}--\eqref{eq:limit-parity-one} imply
\[
 (K_{\sigma,N}^\pm)_{r,r}
 =
 2q_\pm(v_{r,N})+\mathcal O(N^{-1}),
\text{
and }
 (K_{\sigma,N}^\pm)_{r,r+1}
 =
 q_\pm(v_{r,N})+\mathcal O(N^{-1}),
\]
uniformly away from the finitely many endpoint rows already absorbed
into the bounded-rank correction.

The two parity pairings have mesh size \(4\pi/N\).  Hence a component
corresponding to a nonempty interval \(I=[a,b]\) has dimension
\[
 n_{\sigma,N,I}
 =
 \frac{N(b-a)}{4\pi}+\mathcal O(1),
\]
and thus
\[
 v_{r,N}
 =
 a+(b-a)\frac{r}{n_{\sigma,N,I}}
 +\mathcal O(N^{-1})
\]
uniformly in \(r\).  After the affine change
\(u=a+(b-a)x\), the limiting diagonal and off-diagonal profiles are
therefore
\[
 b_I(x)
 =
 2q_\pm(a+(b-a)x),
 \qquad
 a_I(x)
 =
 q_\pm(a+(b-a)x).
\]
Both belong to \(C^1([0,1])\).

Finally, let
\[
 r_{N,+}
 :=
 \#\{0\leq k<M:c_{k,N}>0\},
 \qquad
 r_{N,-}
 :=
 \#\{0\leq k<M:c_{k,N}<0\}.
\]
Since the sampling points have spacing \(2\pi/N\),
\[
 r_{N,+}
 =
 \frac{Nu_0}{2\pi}+\mathcal O(1),
 \qquad
 r_{N,-}
 =
 \frac{N(\pi-u_0)}{2\pi}+\mathcal O(1).
\]
The two parity pairings together use each edge factor once, up to
\(\mathcal O(1)\) endpoint effects.  Consequently
\[
 d_{N,+}
 =
 r_{N,+}+\mathcal O(1)
 =
 \frac{Nu_0}{2\pi}+\mathcal O(1)\text{ and }
 d_{N,-}
 =
 r_{N,-}+\mathcal O(1)
 =
 \frac{N(\pi-u_0)}{2\pi}+\mathcal O(1).
\]
\end{proof}

\begin{lemma}
\label{lem:limit-integral-form}
With the notation of \Cref{lem:limit-Jacobi-reduction}, for every
bounded continuous function \(f:\mathbb C\to\mathbb C\),
\begin{align}
 \lim_{N\to\infty}\int_{\mathbb C} f\,d\mu_{N,\varphi}
 &=
 \frac1{\pi^2}
 \int_0^{u_0}\int_0^\pi
 f\!\left(
 e^{i\varphi/2}
 2\sqrt{c_\varphi(u)}\cos\alpha
 \right)
 \,d\alpha\,du
 \notag\\
 &\quad+
 \frac1{\pi^2}
 \int_{u_0}^{\pi}\int_0^\pi
 f\!\left(
 e^{i(\varphi/2+\pi/2)}
 2\sqrt{-c_\varphi(u)}\cos\alpha
 \right)
 \,d\alpha\,du.
 \label{eq:limit-integral-form}
\end{align}
\end{lemma}

\begin{proof}
The local Jacobi symbol associated with either positive or negative
component in \Cref{lem:limit-Jacobi-reduction} is
\[
 2q_\pm(u)+2q_\pm(u)\cos\eta
 =
 4q_\pm(u)\cos^2(\eta/2),
 \qquad 0\leq\eta\leq\pi.
\]

We first justify replacing \(H_{\sigma,N}\) by
\(H_{\sigma,N}^{(0)}\).  All these real symmetric matrices have
spectra in a common bounded interval.  For a bounded continuous
\(f:\mathbb C\to\mathbb C\), define
\[
 \Phi_f(x)
 :=
 \begin{cases}
 f\!\left(e^{i\varphi/2}\sqrt{x}\right)
 +
 f\!\left(-e^{i\varphi/2}\sqrt{x}\right),
 &x\geq0,\\[1mm]
 f\!\left(e^{i(\varphi/2+\pi/2)}\sqrt{-x}\right)
 +
 f\!\left(-e^{i(\varphi/2+\pi/2)}\sqrt{-x}\right),
 &x<0.
 \end{cases}
\]
The two definitions agree at \(x=0\), so \(\Phi_f\) is continuous.

By \Cref{lem:limit-Jacobi-reduction},
\[
 \rank\bigl(
 H_{\sigma,N}-H_{\sigma,N}^{(0)}
 \bigr)=\mathcal O(1).
\]
Approximate \(\Phi_f\) uniformly on the common spectral interval by a
polynomial.  For a fixed polynomial, the trace difference under a
uniformly bounded rank-\(\mathcal O(1)\) perturbation is
\(\mathcal O(1)\), by the usual telescoping identity.  It follows that
\[
 \frac1N
 \left(
 \tr\Phi_f(H_{\sigma,N})
 -
 \tr\Phi_f(H_{\sigma,N}^{(0)})
 \right)
 \longrightarrow0.
\]
The \(\mathcal O(1)\) exceptional zero eigenvalues from the
odd--even decomposition also disappear after division by \(N\).

We may therefore apply \Cref{prop:Jacobi-symbol} separately to each
tridiagonal component of \(K_{\sigma,N}^\pm\).  Each parity component
has asymptotic row density \(N/(4\pi)\) in the \(u\)-variable.  The two
reflection sectors together have row density \(N/(2\pi)\).
Furthermore, each nonzero eigenvalue \(\mu\) of a two-step block gives
the two square roots $\pm e^{i\varphi/2}\sqrt{\mu}$
when \(\mu>0\), and $\pm e^{i(\varphi/2+\pi/2)}\sqrt{-\mu}$
when \(\mu<0\).

Finally, setting \(\alpha=\eta/2\), the two square-root signs turn
\(\alpha\in[0,\pi/2]\) into the symmetric interval
\([0,\pi]\).  Summing the contributions of all positive and negative
components gives exactly \eqref{eq:limit-integral-form}.
\end{proof}

\begin{lemma}
\label{lem:elliptic-line-densities}
Let $a_{\varphi,+}:=\cos\frac{|\varphi|}{2}$ and $ a_{\varphi,-}:=\sin\frac{|\varphi|}{2}.$
The two measures in \eqref{eq:limit-integral-form} have densities
\[
 g_{\varphi,\pm}(t)
 =
 \frac1{\pi^2}
 \EllK\!\left(
 \sqrt{
 a_{\varphi,\pm}^2-\frac{t^2}{4}
 }
 \right),
 \qquad
 |t|<2a_{\varphi,\pm},
\]
and vanish outside the corresponding intervals.  Their masses are
\[
 \int_{\mathbb R}g_{\varphi,+}(t)\,dt
 =
 1-\frac{|\varphi|}{\pi},
 \qquad
 \int_{\mathbb R}g_{\varphi,-}(t)\,dt
 =
 \frac{|\varphi|}{\pi}.
\]
Their supports are $[-2a_{\varphi,+},2a_{\varphi,+}]$
and $[-2a_{\varphi,-},2a_{\varphi,-}],$
respectively, whenever the corresponding \(a_{\varphi,\pm}\) is
nonzero.
\end{lemma}

\begin{proof}
Fix \(u<u_0\).  In the limiting integral, \(\alpha\) is uniformly
distributed on \([0,\pi]\) with measure \(d\alpha/\pi\).  The map $t=2\sqrt{c_\varphi(u)}\cos\alpha$
pushes this measure forward to the arcsine density $\frac{\mathbf 1_{\{|t|<2\sqrt{c_\varphi(u)}\}}}
 {\pi\sqrt{4c_\varphi(u)-t^2}}\,dt.$
Indeed, this follows directly from the change of variables
\(t=2\sqrt{c_\varphi(u)}\cos\alpha\).
Hence, for \(|t|<2a_{\varphi,+}\),
\[
 g_{\varphi,+}(t)
 =
 \frac1{\pi^2}
 \int_0^{u_+(t)}
 \frac{du}{
 \sqrt{2(\cos\varphi+\cos u)-t^2}
 },
\]
where $u_+(t)
 :=
 \arccos\!\left(
 \frac{t^2}{2}-\cos\varphi
 \right).$

Since $\cos\varphi+\cos u
 =
 2a_{\varphi,+}^2
 -
 2\sin^2\frac u2,$
put $\kappa_+^2
 :=
 a_{\varphi,+}^2-\frac{t^2}{4}.$
With the substitution $\sin\frac u2=\kappa_+\sin\alpha,$
we have $du
 =
 \frac{
 2\kappa_+\cos\alpha
 }{
 \sqrt{1-\kappa_+^2\sin^2\alpha}
 }
 \,d\alpha$
and $\sqrt{
 2(\cos\varphi+\cos u)-t^2
 }
 =
 2\kappa_+\cos\alpha.$
The upper endpoint \(u_+(t)\) corresponds to
\(\alpha=\pi/2\).  Therefore
\[
 g_{\varphi,+}(t)
 =
 \frac1{\pi^2}
 \int_0^{\pi/2}
 \frac{d\alpha}{
 \sqrt{1-\kappa_+^2\sin^2\alpha}
 }
 =
 \frac{\EllK(\kappa_+)}{\pi^2}.
\]

For the second line, set \(v=\pi-u\).  Then
\[
 -\cos\varphi-\cos u
 =
 \cos v-\cos\varphi
 =
 2a_{\varphi,-}^2
 -
 2\sin^2\frac v2.
\]
The same calculation, with
\[
 \kappa_-^2
 :=
 a_{\varphi,-}^2-\frac{t^2}{4},
\]
gives
\[
 g_{\varphi,-}(t)
 =
 \frac{\EllK(\kappa_-)}{\pi^2},
 \qquad
 |t|<2a_{\varphi,-}.
\]

To compute the masses, it is simpler to integrate the conditional
arcsine measures before averaging over \(u\).  Each conditional
arcsine measure has total mass one, so
\[
 \int_{\mathbb R}g_{\varphi,+}(t)\,dt
 =
 \frac{u_0}{\pi}
 =
 1-\frac{|\varphi|}{\pi}
\text{ and }
 \int_{\mathbb R}g_{\varphi,-}(t)\,dt
 =
 \frac{\pi-u_0}{\pi}
 =
 \frac{|\varphi|}{\pi}.
\]
Finally, $2\sqrt{c_\varphi(0)}
 =
 2a_{\varphi,+},
 2\sqrt{-c_\varphi(\pi)}
 =
 2a_{\varphi,-},$
which gives the stated supports.
\end{proof}

\begin{lemma}
\label{lem:limit-polynomial-rate}
For every polynomial \(P(z,\bar z)\),
\[
 \frac1N
 \sum_{z\in\spec A_N(\varphi,0)}
 P(z,\bar z)
 =
 \int_{\mathbb C}P(z,\bar z)\,d\mu_\varphi(z)
 +
 \mathcal O_{P,\varphi}(N^{-1}).
\]
Eigenvalues are counted with algebraic multiplicity.
\end{lemma}

\begin{proof}
By linearity it is enough to consider a monomial $P(z,\bar z)=z^p\bar z^q.$
The constant monomial is exact, so assume $d:=p+q\geq1.$

If \(d\) is odd, every nonzero pair \(z,-z\) contributes zero, while
zero eigenvalues contribute zero.  The corresponding limiting moment
also vanishes because \(\mu_\varphi\) is invariant under
\(z\mapsto-z\).  Thus there is nothing to prove in this case.

Suppose therefore that $d=2m, m\geq1,$
and set $H_N:=H_{+,N}\oplus H_{-,N}.$
The relation between the eigenvalues of the original reflection
blocks and those of \(H_N\) gives
\begin{align}
 \sum_{z\in\spec A_N(\varphi,0)}
 z^p\bar z^q
 &=
 2e^{i(p-q)\varphi/2}
 \tr\!\bigl((H_N)_+^m\bigr)+
 2e^{i(p-q)(\varphi/2+\pi/2)}
 \tr\!\bigl((-H_N)_+^m\bigr)
 +
 \mathcal O(1).
 \label{eq:moment-positive-parts}
\end{align}
Here \(X_+\) denotes the positive part of a real symmetric matrix
\(X\), defined by functional calculus.

Define
\[
 K_N^+
 :=
 K_{+,N}^+\oplus K_{-,N}^+,
 \qquad
 K_N^-
 :=
 K_{+,N}^-\oplus K_{-,N}^-,
\]
and
\[
 H_N^{(0)}
 :=
 K_N^+\oplus(-K_N^-).
\]
By \Cref{lem:limit-Jacobi-reduction},
\[
 \rank(H_N-H_N^{(0)})=\mathcal O(1).
\]
All relevant spectra remain in a common bounded interval, so
\Cref{lem:rank-BV} applied to
\[
 x\longmapsto(x_+)^m,
 \qquad
 x\longmapsto((-x)_+)^m
\]
gives
\begin{align*}
 &
 \left|
 \tr\!\bigl((H_N)_+^m\bigr)
 -
 \tr\!\bigl((H_N^{(0)})_+^m\bigr)
 \right|+
 \left|
 \tr\!\bigl((-H_N)_+^m\bigr)
 -
 \tr\!\bigl((-H_N^{(0)})_+^m\bigr)
 \right|
 =
 \mathcal O(1).
\end{align*}
Since \(K_N^\pm\) are positive semidefinite,
\[
 \tr\!\bigl((H_N^{(0)})_+^m\bigr)
 =
 \tr\!\bigl((K_N^+)^m\bigr) \text{ and }
 \tr\!\bigl((-H_N^{(0)})_+^m\bigr)
 =
 \tr\!\bigl((K_N^-)^m\bigr).
\]

Apply \Cref{lem:quantitative-Jacobi} separately to every connected
tridiagonal component of \(K_N^\pm\).  By
\Cref{lem:limit-Jacobi-reduction}, each component has \(C^1\)
coefficient profiles with error \(\mathcal O(N^{-1})\).  Hence each
unnormalized trace has an \(\mathcal O(1)\) error.  There are only
\(\mathcal O(1)\) components, so after division by \(N\) the total
error is \(\mathcal O(N^{-1})\).

The dimension weights and affine changes of variable are exactly
those used in \Cref{lem:limit-integral-form}, and therefore the main
terms are the corresponding moments of \(\mu_\varphi\).  Substituting
these estimates into \eqref{eq:moment-positive-parts} proves the
claim.
\end{proof}
We have now everything to immediately conclude the proof of the theorem. 

\begin{proof}[Proof of \Cref{thm:intro-limit}]
By \Cref{lem:limit-Jacobi-reduction,lem:limit-integral-form},
the measures \(\mu_{N,\varphi}\) converge weakly to the measure given
by the two integrals in \eqref{eq:limit-integral-form}.
\Cref{lem:elliptic-line-densities} identifies these two measures with
the densities \(g_{\varphi,+}\) and \(g_{\varphi,-}\) in
\eqref{eq:general-line-densities}, and also gives their masses and
supports.  This proves
\eqref{eq:general-explicit-limit} and
\eqref{eq:line-masses-intro}.

Finally, \Cref{lem:limit-polynomial-rate} gives
\eqref{eq:general-polynomial-rate}.
\end{proof}
For the Scottish flag matrix this implies then
\begin{corollary}
\label{cor:full-scottish-limit}
For every \(N\ge3\), define
\[
 \mu_N^{\mathrm{SF}}
 :=\frac1N\sum_{z\in\spec(-B_N/2)}\delta_z,
\]
with algebraic multiplicity.  Then, as \(N\to\infty\) through all
integers,
\[
 \mu_N^{\mathrm{SF}}\Longrightarrow\mu_{\pi/2}.
\]
Moreover, for every polynomial \(P(z,\bar z)\),
\[
 \frac1N\sum_{z\in\spec(-B_N/2)}P(z,\bar z)
 =
 \int_{\C}P(z,\bar z)\,d\mu_{\pi/2}(z)
 +\mathcal O_P(N^{-1}).
\]
\end{corollary}

\begin{proof}
For even \(N\), the result follows from \Cref{thm:intro-limit} and
\eqref{eq:twist-cancellation-intro}.  It remains to consider odd \(N\).

Write $N=2L+1,$
and set $s_{j,N}:=\sin\frac{2\pi j}{N},
    1\le j<N.$
Since \(N\) is odd,
\[
    s_{j,N}>0\quad(1\le j\le L),
    \qquad
    s_{j,N}<0\quad(L+1\le j<N).
\]
Choose square roots \(w_{j,N}^2=s_{j,N}\) by
\[
    w_{j,N}>0\quad(1\le j\le L),
    \qquad
    w_{j,N}\in i(0,\infty)\quad(L+1\le j<N),
\]
and let \(T_N\) be the symmetric zero-diagonal tridiagonal matrix
whose off-diagonal entries are \(w_{1,N},\ldots,w_{N-1,N}\).  Thus
the product of the two entries across its \(j\)-th edge is
\[
    (T_N)_{j,j+1}(T_N)_{j+1,j}
    =w_{j,N}^2
    =\sin\frac{2\pi j}{N}.
\]

We first relate \(T_N\) to the Scottish flag matrix.  Put
\[
    c:=\frac{1}{\sqrt{2i}}
    =\frac{\e^{-i\pi/4}}{\sqrt2}.
\]
Then $c^2=-\frac{i}{2}.$
Hence the edge products of \(cT_N\) are $-\frac{i}{2}\sin\frac{2\pi j}{N}.$
By \Cref{lem:phase-opening} at \(\varphi=\pi/2\), these are exactly
the edge products of the path whose characteristic polynomial is
\(Q_{N,\pi/2}\).  Since the characteristic polynomial of a
zero-diagonal tridiagonal path depends only on its edge products,
\[
    Q_{N,\pi/2}(z)
    =
    \det(zI-cT_N).
\]
Equivalently,
\[
    \det(wI-T_N)
    =
    (\sqrt{2i})^N
    Q_{N,\pi/2}\!\left(\frac{w}{\sqrt{2i}}\right).
\]

For odd \(N\), \Cref{thm:odd-scottish} gives
\[
    \det\!\left(zI+\frac12B_N\right)
    =
    Q_{N,\pi/2}(z).
\]
Consequently,
\begin{equation}
\label{eq:odd-scottish-path-scaling}
    \spec(-B_N/2)
    =
    c\,\spec(T_N)
    =
    \frac{\e^{-i\pi/4}}{\sqrt2}\spec(T_N),
\end{equation}
with algebraic multiplicity.

We next reduce \(T_N\) to a real symmetric Jacobi matrix.  The phases
of the weights \(w_{j,N}\) change only once, between
\(w_{L,N}\) and \(w_{L+1,N}\).  Apply \Cref{prop:odd-even} and choose,
as in \Cref{thm:mirrored-path}, the one of the two matrices
\(B^TB\) and \(BB^T\) whose off-diagonal entries do not contain the
product
\[
    w_{L,N}w_{L+1,N}.
\]
Call the resulting matrix \(K_N\).  Because every product of
consecutive weights occurring off the diagonal of \(K_N\) involves
weights having the same phase, \(K_N\) is real symmetric.  Its
dimension satisfies
\[
    d_N:=\dim K_N=\frac N2+\mathcal O(1).
\]

By \Cref{prop:odd-even}, the nonzero eigenvalues of \(T_N\) are
obtained from those of \(K_N\) by taking both square roots:
\[
    w^2=\mu,
    \qquad
    \mu\in\spec(K_N).
\]
Depending on which of \(B^TB\) and \(BB^T\) was chosen, there may be
one additional zero eigenvalue; this will have no effect on the
normalized limit.

Conjugating \(K_N\) by a real diagonal matrix with entries
\(\pm1\), we may arrange that all its off-diagonal entries are
nonnegative.  Denote the resulting real symmetric Jacobi matrix by
\(H_N\).  This conjugation does not change its eigenvalues.

We now identify the slowly varying coefficients of \(H_N\).  By
\eqref{eq:gram-entries}, its diagonal entries, away from a bounded
number of endpoint and central rows, are sums of two consecutive
numbers \(s_{j,N}\), whereas the absolute values of its
off-diagonal entries are of the form $\sqrt{|s_{j,N}s_{j+1,N}|}.$
Since consecutive sampling points differ by \(2\pi/N\), and since
\(d_N=N/2+\mathcal O(1)\), the \(r\)-th row corresponds to a point
\[
    x_{r,N}=\frac{r}{d_N}+\mathcal O(N^{-1}).
\]
It follows that, uniformly away from the bounded number of exceptional
rows,
\[
    (H_N)_{r,r}
    =
    2\sin(2\pi x_{r,N})+\mathcal O(N^{-1}),
\]
and
\[
    (H_N)_{r,r+1}
    =
    |\sin(2\pi x_{r,N})|+\mathcal O(N^{-1}).
\]
Changing the bounded number of exceptional rows gives a Jacobi matrix
\(\widetilde H_N\) such that
\[
    \rank(H_N-\widetilde H_N)=\mathcal O(1)
\]
and whose coefficients converge uniformly to
\[
    b(x)=2\sin(2\pi x),
    \qquad
    a(x)=|\sin(2\pi x)|,
    \qquad 0\le x\le1.
\]

We can now compute the weak limit.  Let
\(f:\mathbb C\to\mathbb C\) be bounded and continuous, and define a
continuous function \(\Phi_f:\mathbb R\to\mathbb C\) by
\[
    \Phi_f(\mu)
    :=
    \begin{cases}
    f(c\sqrt{\mu})+f(-c\sqrt{\mu}),
        &\mu\ge0,\\[1mm]
    f(ci\sqrt{-\mu})+f(-ci\sqrt{-\mu}),
        &\mu<0.
    \end{cases}
\]
The two definitions agree at \(\mu=0\).  The square-root relation
above and \eqref{eq:odd-scottish-path-scaling} imply
\[
    \int f\,d\mu_N^{\mathrm{SF}}
    =
    \frac1N\tr\Phi_f(H_N)+o(1).
\]
The \(o(1)\) only accounts for the bounded number of exceptional zero
modes.

Since \(H_N-\widetilde H_N\) has uniformly bounded rank and all the
spectra remain in a common compact interval, uniform polynomial
approximation and trace telescoping give
\[
    \frac1N
    \left(
    \tr\Phi_f(H_N)-\tr\Phi_f(\widetilde H_N)
    \right)
    \longrightarrow0.
\]
We may therefore apply \Cref{prop:Jacobi-symbol} to
\(\widetilde H_N\).  Since \(d_N/N\to1/2\),
\begin{equation}
\label{eq:odd-scottish-prelimit}
    \lim_{N\to\infty}\int f\,d\mu_N^{\mathrm{SF}}
    =
    \frac1{2\pi}
    \int_0^1\int_0^\pi
    \Phi_f\!\left(
        2\sin(2\pi x)
        +2|\sin(2\pi x)|\cos\eta
    \right)
    d\eta\,dx.
\end{equation}

It remains to identify this expression.  Put $s(x):=\sin(2\pi x).$
For \(0<x<1/2\), one has \(s(x)>0\), and the local symbol in
\eqref{eq:odd-scottish-prelimit} is $ 2s(x)+2s(x)\cos\eta
    =
    4s(x)\cos^2(\eta/2).$
Taking both square roots and then multiplying by \(c\) gives the two
points
\[
    \pm\e^{-i\pi/4}
    \sqrt{2s(x)}\cos(\eta/2).
\]
Using the two signs to extend \(\eta/2\in[0,\pi/2]\) to
\(\alpha\in[0,\pi]\), the contribution of this half of the interval
is
\begin{equation}
\label{eq:odd-positive-half}
    \frac1\pi
    \int_0^{1/2}\int_0^\pi
    f\!\left(
        \e^{-i\pi/4}
        \sqrt{2\sin(2\pi x)}\cos\alpha
    \right)
    d\alpha\,dx.
\end{equation}

For \(1/2<x<1\), one has \(s(x)<0\), and the local symbol is
\[
    2s(x)+2|s(x)|\cos\eta
    =
    -4|s(x)|\sin^2(\eta/2).
\]
Its two square roots are purely imaginary.  Since $ci=\frac{\e^{i\pi/4}}{\sqrt2},$
the contribution of this half is
\begin{equation}
\label{eq:odd-negative-half}
    \frac1\pi
    \int_{1/2}^{1}\int_0^\pi
    f\!\left(
        \e^{i\pi/4}
        \sqrt{2|\sin(2\pi x)|}\cos\alpha
    \right)
    d\alpha\,dx.
\end{equation}

We finally rewrite these two integrals in the variables used in
\eqref{eq:general-explicit-limit}.  In
\eqref{eq:odd-negative-half}, set \(v=2\pi x\) and fold
\([\pi,2\pi]\) about \(3\pi/2\).  With $u=|v-3\pi/2|\in[0,\pi/2]$
we have $|\sin v|=\cos u,$
and every \(u\in(0,\pi/2)\) has two preimages.  Hence
\[
    \eqref{eq:odd-negative-half}
    =
    \frac1{\pi^2}
    \int_0^{\pi/2}\int_0^\pi
    f\!\left(
        \e^{i\pi/4}
        \sqrt{2\cos u}\cos\alpha
    \right)
    d\alpha\,du.
\]

Similarly, in \eqref{eq:odd-positive-half}, set \(v=2\pi x\), fold
\([0,\pi]\) about \(\pi/2\), and then put
\[
    u=\pi-|v-\pi/2|\in[\pi/2,\pi].
\]
Then $\sin v=-\cos u.$
Moreover $\e^{-i\pi/4}=-\e^{3i\pi/4},$
and the minus sign is absorbed by the substitution
\(\alpha\mapsto\pi-\alpha\).  Therefore
\[
    \eqref{eq:odd-positive-half}
    =
    \frac1{\pi^2}
    \int_{\pi/2}^{\pi}\int_0^\pi
    f\!\left(
        \e^{3i\pi/4}
        \sqrt{-2\cos u}\cos\alpha
    \right)
    d\alpha\,du.
\]
The last two displays are exactly the two integrals in
\eqref{eq:general-explicit-limit} at \(\varphi=\pi/2\).  Thus
\[
    \mu_N^{\mathrm{SF}}
    \Longrightarrow
    \mu_{\pi/2}
\]
along the odd subsequence as well.

It remains to prove the quantitative statement for polynomial
moments.  By linearity it is enough to take
\[
    P(z,\bar z)=z^p\bar z^q.
\]
If \(p+q\) is odd, the two eigenvalues \(z\) and \(-z\) cancel
exactly, and the limiting moment vanishes for the same reason.  The
constant monomial is exact.

Suppose therefore that
\[
    p+q=2r,\qquad r\ge1.
\]
For real \(\mu\), the contribution of the two square roots is
\[
    \Phi_P(\mu)
    =
    2c^p\bar c^{\,q}(\mu_+)^r
    +
    2(ci)^p\overline{(ci)}^{\,q}((-\mu)_+)^r,
\]
where \(\mu_+=\max\{\mu,0\}\).  Thus the polynomial moment is a fixed
linear combination of positive-part powers of \(H_N\).

Cut the Jacobi matrix at \(x=1/2\).  This changes only a bounded
number of rows.  On each of the two resulting components the
coefficient profiles
\[
    2\sin(2\pi x),
    \qquad
    |\sin(2\pi x)|
\]
are \(C^1\), and the coefficient errors are
\(\mathcal O(N^{-1})\).  Hence
\Cref{lem:quantitative-Jacobi} gives an \(\mathcal O(1)\) error for
each unnormalized trace.  The bounded-rank changes contribute only
\(\mathcal O(1)\) by \Cref{lem:rank-BV}.  There are only two
components, so the total unnormalized error is \(\mathcal O(1)\).

After division by \(N\), the error is therefore
\(\mathcal O_P(N^{-1})\).  The same changes of variables used above
identify the main term with
\[
    \int_{\mathbb C}P(z,\bar z)\,d\mu_{\pi/2}(z).
\]
This proves the stated polynomial-moment estimate for odd \(N\).
Combining the odd and even subsequences completes the proof.
\end{proof}

\appendix
\section{Proof of the determinant formula under the reversal identity}
\label{app:mirrored}

We prove \Cref{thm:mirrored-path}.  Assume \eqref{eq:mirrored-rho} and
choose the square roots
\begin{equation}
\label{eq:mirrored-weights}
 w_j=\e^{i\theta}\sqrt{a_j},
 \qquad
 w_{2s+1-j}=i\e^{i\theta}\sqrt{a_j},
 \qquad 1\le j\le s,
\end{equation}
so that \(w_{2s+1-j}=iw_j\).  Apply \Cref{prop:odd-even} with \(n=2s+1\); the
block \(B\in\C^{(s+1)\times s}\) and the two matrices \(B^TB\) and
\(BB^T\) are the ones displayed there.

The sign change occurs at the single pair \(w_sw_{s+1}\).  If
\(s\) is odd, this pair occurs among the products \(w_1w_2,w_3w_4,\ldots\), so
\(B^TB\) avoids it; if \(s\) is even, it occurs among
\(w_2w_3,w_4w_5,\ldots\), so \(BB^T\) avoids it.  Every diagonal entry in
\eqref{eq:gram-entries} is a sum of squares \(w_j^2=\pm\e^{2i\theta}a_j\) and
therefore also lies on \(\e^{2i\theta}\R\).  All entries of the selected two-step
matrix \(K\) hence have the common phase \(\e^{2i\theta}\) times a real
number, which is \eqref{eq:K-H}, with \(d=s\) for \(s\) odd and \(d=s+1\) for
\(s\) even.  In both cases \(d\) is odd.  The off-diagonal entries of \(K\)
are products of two consecutive nonzero weights, so every off-diagonal entry of \(H\) is nonzero.

To see the reflection sign, note that \eqref{eq:mirrored-weights} gives
\begin{align*}
 w_{2s+1-j}^2&=-w_j^2
 &&(1\le j\le s),\\
 w_{2s-j}w_{2s+1-j}&=-w_jw_{j+1}
 &&(1\le j<s).
\end{align*}
By \eqref{eq:gram-entries}, these identities reverse the diagonal and
off-diagonal entries of \(H\) and change their signs.  If \(R_d\) denotes
coordinate reversal on \(\R^d\), then \(H\) is odd under coordinate reversal:
\begin{equation}
\label{eq:anti-centro}
 R_dHR_d=-H.
\end{equation}

Because \(H\) is real symmetric and tridiagonal, with every off-diagonal entry nonzero, its spectrum is
real and simple.  The identity \eqref{eq:anti-centro} pairs every eigenvalue
\(\mu\) with \(-\mu\).  Since \(d\) is odd,
\[
 \spec(H)=\{0,\pm\nu_1,\ldots,\pm\nu_m\},
 \qquad 0<\nu_1<\cdots<\nu_m,
\]
with \(d=2m+1\).

Suppose first that \(s=2m+1\).  Then \(K=B^TB\), and \eqref{eq:Schur} and
\eqref{eq:K-H} give
\begin{align*}
 \det(zI-T)
 &=z\det(z^2I_s-\e^{2i\theta}H)=z\,z^2\prod_{r=1}^m
   (z^2-\e^{2i\theta}\nu_r)
   (z^2+\e^{2i\theta}\nu_r)=z^3\prod_{r=1}^m
   (z^4-\e^{4i\theta}\nu_r^2).
\end{align*}
This is \eqref{eq:mirror-factor-odd}.  Now let \(s=2m\).  Then
\(K=BB^T\), and the second identity in \eqref{eq:Schur},
\[
 \det(tI_{s+1}-BB^T)=t\det(tI_s-B^TB),
\]
shows that, initially for \(z\ne0\),
\begin{align*}
 \det(zI-T)
 &=z\det(z^2I_s-B^TB)=z^{-1}\det(z^2I_{s+1}-BB^T)\\
 &=z^{-1}\det(z^2I_{s+1}-\e^{2i\theta}H)=z\prod_{r=1}^m
   (z^4-\e^{4i\theta}\nu_r^2),
\end{align*}
which is \eqref{eq:mirror-factor-even}.  Each nonzero root satisfies
\begin{equation*}
 z^2=\pm\e^{2i\theta}\nu_r \text{ and hence }
 z\in\e^{i\theta}\R\text{ or }
 z\in\e^{i(\theta+\pi/2)}\R.
\end{equation*}
This proves \eqref{eq:mirror-cross-conclusion}.

The preceding identity extends to \(z=0\) by polynomial continuation.
Finally, the symmetric tridiagonal matrix in \eqref{eq:symmetrization}
has every off-diagonal entry nonzero, so its eigenspace at zero is
one-dimensional by the same
recurrence argument.  The algebraic zero multiplicities in the two determinant
formulas are respectively one and three.  Therefore the Jordan blocks at zero are
\(J_1(0)\) and \(J_3(0)\), as claimed.
\qed

\section{Jordan blocks at zero when
\texorpdfstring{\(\varphi=\pi/2\)}{the phase is pi/2}}
\label{app:scottish-zero}

We complete the zero-eigenvalue calculation used in
\Cref{thm:intro-scottish}.  Write \(N=2M\), and let \(T_{+,N}\) and
\(T_{-,N}\) be the \(+1\)-reflection and \(-1\)-reflection blocks of
\Cref{sub:aligned}.

If \(N=4m\), then $\dim T_{+,N}=2m+1$ and $\dim T_{-,N}=2m-1.$
Both tridiagonal matrices have all off-diagonal entries nonzero and odd dimension, hence each has a
one-dimensional kernel.  In the notation of \Cref{thm:mirrored-path}, their
parameters are \(s_+=m\) and \(s_-=m-1\).  Exactly one of these integers is
odd.  The determinant formula proved in \Cref{thm:mirrored-path} therefore gives zero factors
\(z^3\) and \(z\) in the two sectors.  Since the geometric multiplicity is
one in each sector, the corresponding Jordan blocks are \(J_3(0)\) and
\(J_1(0)\).  Thus
\[
 A_N|_{\cG_0(A_N)}\sim J_3(0)\oplus J_1(0).
\]

The case \(N=4m+2\) is different.  There one of the two entries across an
edge vanishes, so that edge is traversed in one direction only and the sector
becomes block triangular.  The zero structure of such a coupling is the
content of the following lemma.

\begin{lemma}\label{lem:one-directional}
Let \(J_L,J_R\in\mathbb R^{d\times d}\) be symmetric tridiagonal matrices
with zero diagonal and all off-diagonal entries nonzero, where \(d\) is
odd, and let \(\gamma\ne0\).
For arbitrary \(\theta_L,\theta_R\in\R\), set
\begin{equation*}
 \mathcal T=
 \begin{pmatrix}
  \e^{i\theta_L}J_L&\gamma e_de_1^T\\
  0&\e^{i\theta_R}J_R
 \end{pmatrix}.
\end{equation*}
Then zero has algebraic multiplicity two, geometric multiplicity one, and
Jordan form \(J_2(0)\).  
\end{lemma}

\begin{proof}
Write \(A=\e^{i\theta_L}J_L\),
\(D=\e^{i\theta_R}J_R\), and
\(\Gamma=\gamma e_de_1^T\).
A zero-diagonal real symmetric tridiagonal matrix of odd size has a zero eigenvalue because its
characteristic polynomial is odd.  Because every off-diagonal entry is nonzero, that eigenvalue is simple.
Choose nonzero vectors
\begin{equation*}
 J_Lv_L=0,
 \qquad J_Rv_R=0.
\end{equation*}
Every endpoint coordinate of \(v_L\) and \(v_R\) is nonzero.  Indeed,
if the first coordinate vanished, the first row and the three-term
recurrence would force the vector to vanish successively from the left.
The same argument from the last row proves the assertion for the last
coordinate.
Block triangularity gives
\begin{equation*}
 \det(zI-\mathcal T)=\det(zI-A)\det(zI-D),
\end{equation*}
so zero has algebraic multiplicity two.  If
\(\mathcal T(x,y)^T=0\), then
\begin{equation*}
 Dy=0,
 \qquad Ax+\Gamma y=0.
\end{equation*}
The first equation gives \(y=cv_R\), and the second becomes
\begin{equation*}
 Ax=-\gamma c(v_R)_1e_d.
\end{equation*}
Because \(A^T=A\) and \(Av_L=0\), pairing with \(v_L\) yields
\begin{equation*}
 0=v_L^TAx
 =-\gamma c(v_L)_d(v_R)_1.
\end{equation*}
All three factors other than \(c\) are nonzero, hence \(c=0\).  Thus
\begin{equation*}
 \ker\mathcal T=\operatorname{span}\{(v_L,0)^T\},
 \qquad \dim\ker\mathcal T=1.
\end{equation*}

Since the algebraic multiplicity is two and the geometric multiplicity is
one, the Jordan block for zero is \(J_2(0)\).  The transpose case follows by
interchanging left and right.
\end{proof}

Now let \(N=4m+2\), so \(M=2m+1\).  In the quadratic-phase basis, the central
edge joins the indices \(m\) and \(m+1\).  Since $\omega^{m+1/2}=\e^{\pi i/2}=i,$
the two entries across this edge are
\[
 \alpha_{m,N}=\frac{1+i\omega^{m+1/2}}2=0,
 \qquad
 \beta_{m+1,N}=\frac{1+i\omega^{-m-1/2}}2=1.
\]
The zero-product edge occurs at the middle of both blocks obtained from the reflection.  In the
\(+1\)-reflection block \(T_{+,N}\), of dimension \(2m+2\), it separates two
diagonal path blocks of size \(m+1\); in the \(-1\)-reflection block
\(T_{-,N}\), of dimension \(2m\), it separates two diagonal path blocks of
size \(m\).  The remaining directed entry makes each folded matrix block
triangular rather than a direct sum.  Thus, if \(m\) is even, the singular
sector is \(T_{+,N}\) and its two diagonal blocks have the odd size
\(d=m+1\), while \(T_{-,N}\) has even diagonal blocks of size \(m\) and is
invertible at zero.  If \(m\) is odd, the roles are reversed: the singular
sector is \(T_{-,N}\) with odd block size \(d=m\), and \(T_{+,N}\) has even
diagonal blocks of size \(m+1\) and is invertible at zero.  The
invertibility follows from
\[
 \det J_{2r}
 =(-1)^r\prod_{j=1}^{r}(J_{2j-1,2j})^2\ne0
\]
for a zero-diagonal symmetric tridiagonal matrix with nonzero
off-diagonal entries.

After separate diagonal similarity to a symmetric matrix, let
\(J_L,J_R\in\R^{d\times d}\) denote the resulting symmetric tridiagonal matrices with zero diagonal and all off-diagonal entries nonzero, where \(d\) is the odd dimension specified
above.  The block containing the zero eigenvalue is similar, up to transpose, to
\[
 \begin{pmatrix}
  \e^{\pi i/4}J_L&\gamma e_de_1^T\\
  0&\e^{-\pi i/4}J_R
 \end{pmatrix}.
\]
Each diagonal block has a
simple zero eigenvalue, where \(\gamma\ne0\).  By
\Cref{lem:one-directional}, zero has algebraic multiplicity two and geometric
multiplicity one in this sector.  Hence
\[
 A_N|_{\cG_0(A_N)}\sim J_2(0).
\]

\section{The zero eigenspace when
\texorpdfstring{\(\tau=-1\)}{the twist is -1}}
\label{app:antiperiodic-zero}

\begin{proposition}
\label{prop:antiperiodic-zero}
Let \(N=4m+2\ge6\).  The zero eigenspace of the matrix
\[
 A_N^{-1}(\pi/2,\pi/2)
\]
has dimension two.  Since its algebraic multiplicity is also two, both
zero Jordan blocks are one-dimensional.
\end{proposition}

\begin{proof}
Put \(M=N/2=2m+1\), \(\delta=\pi/N\), and
\(U=\diag(\e^{i\delta j})_{j=0}^{N-1}\).  Then
\[
 U^{-1}S_{-1}U=\e^{-i\delta}S,
\]
so
\[
 \widehat A:=U^{-1}A_N^{\tau=-1}(\pi/2,\pi/2)U
 =D_N(\pi/2)+\frac{i}{2}
  \bigl(\e^{-i\delta}S+\e^{i\delta}S^{-1}\bigr).
\]
In the periodic quadratic-phase basis $\psi_k(j)=N^{-1/2}\exp\!\left(\frac{\pi i j^2}{N}\right)\omega^{jk},$
define
$ \alpha_k:=\frac{i}{2}(1+\omega^{k+1}), \beta_k:=\frac{i}{2}(-1+\omega^{-k}).$ Then
\[
 \widehat A\psi_k=\alpha_k\psi_{k+1}+\beta_k\psi_{k-1}.
\]
Here \(\alpha_{M-1}=0\) and \(\beta_0=0\), while
\(\beta_M=-i\) and \(\alpha_{N-1}=i\).  Set $\Gamma:=i e_1e_M^T-i e_Me_1^T.$
Ordering first \(\psi_0,\ldots,\psi_{M-1}\) and then
\(\psi_M,\ldots,\psi_{N-1}\) gives a block triangular matrix
\[
 \widehat A=
 \begin{pmatrix}T_L&\Gamma\\0&T_R\end{pmatrix},
\]
which defines the diagonal blocks \(T_L\) and \(T_R\).  Both are
\(M\times M\) zero-diagonal tridiagonal matrices with every
off-diagonal entry nonzero; \(M\) is odd.  Their products \(\rho_j\)
satisfy
\[
 \rho_k=\frac{i}{2}\sin\frac{2\pi(k+1)}N\in i(0,\infty)
 \quad(0\le k\le M-2)
\]
for \(T_L\), and \(\rho_k\in-i(0,\infty)\) for
\(M\le k\le N-2\) in \(T_R\).  Thus, by \Cref{lem:diagonal-similarity}, each
block is diagonally similar to a
phase multiple of a real symmetric tridiagonal matrix with every off-diagonal entry nonzero, and hence has a simple zero eigenvalue.

Choose nonzero vectors \(u_L,v_L,v_R\) with
\[
 u_L^TT_L=0,\qquad T_Lv_L=0,\qquad T_Rv_R=0,
\]
and normalize \(u_L,v_R\) by \((u_L)_1=(v_R)_1=1\).  The zero-diagonal recurrences show that coordinates
\(2,4,\ldots,M-1\) vanish and give
\begin{align*}
 (u_L)_M
 &=(-1)^m\prod_{r=0}^{m-1}
   \frac{\beta_{2r+1}}{\alpha_{2r+1}}=1 \text{ and }
 (v_R)_M=(-1)^m\prod_{r=0}^{m-1}
   \frac{\alpha_{M+2r}}{\beta_{M+2r+2}}=1.
\end{align*}
Indeed,
\[
 \alpha_k=i\e^{\pi i(k+1)/N}\cos\frac{\pi(k+1)}N,
 \qquad
 \beta_k=\e^{-\pi i k/N}\sin\frac{\pi k}N,
\]
and in both products the phases cancel, while
\[
 \left\{\cos\frac{(2r+2)\pi}{N}:0\le r<m\right\}
 =\left\{\sin\frac{(2r+1)\pi}{N}:0\le r<m\right\}
\]
as multisets.  Consequently,
\[
 u_L^T\Gamma v_R=i-i=0.
\]
The equation \(T_Lx=-\Gamma v_R\) is therefore solvable, because the left
nullspace of \(T_L\) is spanned by \(u_L\).  Hence
\((x,v_R)^T\) is a zero eigenvector independent of
\((v_L,0)^T\).  The dimension of the nullspace is at least two.  By
\eqref{eq:twist-cancellation-intro} and \Cref{thm:intro-scottish}, the
algebraic multiplicity is exactly two, so the dimension of the nullspace is exactly two and
the zero Jordan form is \(J_1(0)\oplus J_1(0)\).
\end{proof}
\FloatBarrier

\end{document}